\documentclass[11pt]{article}

\usepackage[T1]{fontenc}
\usepackage[utf8]{inputenc}
\usepackage{lmodern}
\usepackage[margin=1in]{geometry}
\usepackage{amsmath,amssymb,amsthm,bm}
\usepackage{graphicx}
\usepackage{booktabs,array,multirow,makecell}
\usepackage{caption}
\usepackage{algorithm}
\usepackage{csquotes}
\usepackage{natbib}
\usepackage{url}
\usepackage[colorlinks=true,citecolor=blue,linkcolor=blue,urlcolor=blue]{hyperref}

\newtheorem{theorem}{Theorem}
\newtheorem{lemma}{Lemma}

\newtheorem{assumptionA}{Assumption}

\newtheorem{assumptionB}{Assumption}

\newtheorem{conditionC}{Condition}

\providecommand{\lesssim}{\mathrel{\raisebox{-.35ex}{$\stackrel{\scriptstyle <}{\scriptstyle \sim}$}}}
\providecommand{\gtrsim}{\mathrel{\raisebox{-.35ex}{$\stackrel{\scriptstyle >}{\scriptstyle \sim}$}}}

\title{Deep Neural Networks for Doubly Robust Estimation with Nonprobability Survey Samples}
\author{
Yufang Dai\\
\small School of Statistics and Data Science, Jiangxi University of Finance and Economics, Nanchang, China\\
\small Department of Mathematics and Statistics, University of Calgary, Calgary, Canada
\and
Shihua Luo\\
\small School of Statistics and Data Science, Jiangxi University of Finance and Economics, Nanchang, China
\and
Wendy Lou\\
\small Division of Biostatistics, Dalla Lana School of Public Health, University of Toronto, Canada
\and
Zilin Wang\\
\small Department of Mathematics, Wilfrid Laurier University, Waterloo, Canada
\and
Xuewen Lu\\
\small Department of Mathematics and Statistics, University of Calgary, Calgary, Canada\\
\small \href{mailto:xlu@ucalgary.ca}{xlu@ucalgary.ca}
}
\date{}

\begin{document}
\maketitle

\begin{abstract}
Integrating probability and nonprobability survey samples is an important problem in modern survey sampling. Nonprobability samples often contain rich outcome information but may lack population representativeness, whereas probability samples provide design-based auxiliary information but may not contain the study variable. We propose a deep neural network (DNN)-assisted doubly robust framework for estimating the finite population mean from these two data sources. The proposed method models the logit sampling score for the nonprobability sample as an unknown nonparametric function and estimates it by maximizing a pseudo-likelihood that combines information from the nonprobability sample and a reference probability sample. The DNN parameters are optimized using the ADAM algorithm. The resulting DNN-estimated sampling scores are incorporated into a DNN-assisted inverse-probability weighted estimator and a deep doubly robust estimator. We establish consistency and convergence rates under regularity conditions and evaluate the finite-sample performance of the proposed estimators through simulation studies and an empirical application using Pew Research Center and Behavioral Risk Factor Surveillance System data. The results suggest that the proposed estimators can improve robustness to parametric propensity-score misspecification, especially when the true selection mechanism is nonlinear.
\end{abstract}



\section*{Résumé}
L'int\'egration d'\'echantillons probabilistes et non probabilistes est un probl\`eme important dans l'\'echantillonnage moderne. Les \'echantillons non probabilistes contiennent souvent une information riche sur la variable d'int\'er\^et, mais peuvent manquer de repr\'esentativit\'e dans la population, tandis que les \'echantillons probabilistes fournissent une information auxiliaire fond\'ee sur le plan de sondage sans toujours contenir la variable d'int\'er\^et. Nous proposons un cadre doublement robuste assist\'e par r\'eseau neuronal profond pour estimer la moyenne d'une population finie \`a partir de ces deux sources de donn\'ees. La m\'ethode propos\'ee mod\'elise le logit du score d'\'echantillonnage de l'\'echantillon non probabiliste comme une fonction non param\'etrique inconnue, estim\'ee par maximisation d'une pseudo-vraisemblance combinant l'information provenant de l'\'echantillon non probabiliste et d'un \'echantillon probabiliste de r\'ef\'erence. Les param\`etres du r\'eseau neuronal profond sont optimis\'es au moyen de l'algorithme ADAM. Les scores d'\'echantillonnage ainsi estim\'es sont ensuite int\'egr\'es dans un estimateur assist\'e par r\'eseau neuronal profond et pond\'er\'e par l'inverse de la probabilit\'e, ainsi que dans un estimateur doublement robuste profond. Nous \'etablissons la coh\'erence et les taux de convergence sous des conditions de r\'egularit\'e, puis \'evaluons les performances en \'echantillon fini des estimateurs propos\'es au moyen d'\'etudes de simulation et d'une application empirique utilisant les donn\'ees du Pew Research Center et du Behavioral Risk Factor Surveillance System. Les r\'esultats sugg\`erent que les estimateurs propos\'es peuvent am\'eliorer la robustesse face \`a une mauvaise sp\'ecification param\'etrique du score de propension, en particulier lorsque le m\'ecanisme de s\'election v\'eritable est non lin\'eaire.

\section{Introduction} \label{sec:introduction}        % example of a section label

Probability sampling has long been the standard approach for data collection in official statistics and social research because probability-based sampling designs support valid inference for finite population quantities. Under known randomization mechanisms, probability samples provide a principled basis for estimating population characteristics and are often regarded as the gold standard for finite population inference. However, changes in social structure, communication patterns, and survey participation have created major challenges for traditional probability surveys, including low response rates, escalating costs, and time-consuming data collection \citep{aut2016}. To address these challenges, an increasing number of studies have utilized multiple data sources such as web-based survey panels, satellite information, and mobile sensor data generated by the development of internet technology and the popularity of smart devices. Such data are collectively referred to as nonprobability samples, and \cite{aut2015} stated that nonprobability survey data serve as an alternative and complement to official statistical survey data. However, because nonprobability samples are not selected according to a known probability sampling design, they may suffer from coverage error and selection bias, which complicates inference for finite population characteristics. Since probability and nonprobability samples have complementary strengths and limitations, it is natural to integrate information from both sources when estimating population quantities.

Existing methods for data integration can be broadly categorized into three classes. The first approach is calibration weighting. 
\cite{aut1992} first proposed the concept of calibration estimation; when population totals or means of auxiliary variables are known, calibration weights can be constructed using this auxiliary information.
\cite{aut2001} proposed a model-calibration estimator by constructing a superpopulation model between the study variable and auxiliary variables and incorporating full auxiliary information into calibration estimation. This technique forces the moments or the empirical distribution of auxiliary variables to be the same between the probability sample and the nonprobability sample, so that after calibration the weighted distribution in the nonprobability sample appears similar to that in the target population \citep{aut2011a}. Nevertheless, a critical limitation of this approach lies in its reliance on prior knowledge of population-level information, which is often unavailable in practical applications. 

The second approach is mass imputation. In this framework, the nonprobability sample is used as a training dataset to build a prediction model for the study variable, and the fitted model is subsequently applied to estimate the unobserved study variable for each unit in the probability sample. 
\cite{aut2007} proposed using the value of the nearest neighbor for mass imputation, but did not discuss its properties theoretically. \cite{aut2021} proposed using regression models for mass imputation and discussed its statistical properties, including consistent variance estimation. However, such a parametric mass imputation method is subject to model misspecification bias. 

The third approach is propensity score adjustment. In this approach, the probability of a unit being selected into the nonprobability sample, which is referred to as the propensity or sampling score, is modeled and estimated for all units in the nonprobability sample. \cite{aut2011} estimated participation rates by fitting a logistic regression model to the combined nonprobability sample and probability sample. Sample weights for the probability sample were scaled by a constant so that the scaled probability sample was assumed to represent the complement of the nonprobability sample. Each unit in the nonprobability sample was assigned a weight of one. The results show that the sum of the scaled weights of the probability and nonprobability combined samples is an estimate of the population size, but the estimator is biased especially when the participation rate of the nonprobability sample is large. \cite{aut2020} estimated the participation rate by manipulating the log-likelihood estimating equation. The resulting estimator is consistent and approximately unbiased regardless of the magnitude of participation rates. 

However, both mass imputation and propensity score adjustment rely on model specification, and model misspecification can lead to biased estimation. In order to enhance the efficiency and robustness of estimation, \cite{aut2020} proposed doubly robust (DR) estimators for the finite population mean by combining an estimated propensity score with an outcome regression model. While DR estimators guarantee consistency if at least one of the two models is correctly specified, they remain susceptible to bias when both models are misspecified. In their method, the relationship between $Y$ and $X$ is restricted to a prespecified parametric functional form, typically involving linear components. However, parametric procedures may suffer from bias if the functional form is misspecified or if the vector $X$ fails to include relevant interactions or predictors. In contrast, nonparametric methods have the ability to capture nonlinear trends and tend to be more robust to functional-form misspecification. In the last decade, 
machine learning methods have become increasingly competitive tools for modeling complex nonlinear relationships. Compared with classical nonparametric methods such as kernel smoothing, deep learning methods can be more scalable in high-dimensional settings when the target function has suitable smoothness or structural properties.

Among machine learning methods, deep neural networks (DNNs) provide a flexible framework for approximating complex nonlinear functions. Multilayer feedforward neural networks can approximate broad classes of continuous functions \citep{aut1993}. Under some smoothness and structural assumptions, \cite{aut2020b} showed that DNN estimators can circumvent the curse of dimensionality and achieve the optimal minimax rate of convergence. 
These properties make DNNs attractive for modeling participation mechanisms in nonprobability samples, where the relationship between auxiliary variables and sample inclusion may involve nonlinearities and interactions. However, DNN-based sampling-score estimation has not been fully developed for doubly robust inference with integrated probability and nonprobability survey samples.

%%%%%%%%
To fill this gap, we propose a DNN-assisted doubly robust estimator for the mean of a finite population and study its asymptotic properties under a sequence of growing finite populations. The main contributions of this paper are threefold. First, we develop a DNN-based sampling-score estimator by modeling the logit participation probability as an unknown nonparametric function and estimating it through a pseudo-likelihood that combines information from the nonprobability sample and a reference probability sample. Second, we incorporate the resulting DNN-estimated sampling scores into inverse-probability weighting and doubly robust estimation, leading to a deep doubly robust estimator, denoted DDR, that is designed to reduce sensitivity to parametric propensity-score misspecification while retaining the structure of doubly robust estimation. Third, we establish consistency and convergence rates for the proposed estimators under regularity conditions and evaluate their finite-sample performance through simulations and an empirical application. The simulation results suggest that the DDR estimator can substantially reduce bias and mean squared error relative to conventional parametric IPW and DR estimators, especially when the true selection mechanism contains nonlinear terms that are omitted from the parametric propensity-score model.

%%%%%%%%%

The paper proceeds as follows. 
\autoref{sec:Basic set-up} introduces the data-integration setup and reviews three commonly used estimators for the finite population mean. 
\autoref{sec:Methodology and asymptotic property} presents the proposed DNN-based sampling-score estimator, the ADAM algorithm used to optimize the loss function, the DIPW and DDR estimators, and the main asymptotic results. 
\autoref{sec:Simulation studies} reports simulation studies evaluating the finite-sample performance of the proposed estimators. 
\autoref{sec:Application} presents an empirical application using a nonprobability survey sample collected by the Pew Research Center (PRC), together with auxiliary information from the Behavioral Risk Factor Surveillance System (BRFSS). 
\autoref{sec:Conclusion} provides concluding remarks and discussion. 
\autoref{sec:Appendix} contains the regularity conditions and technical proofs.

%%%%%%%%%%%%%%%%%%%%%%%%%%%
%% Additional sections
\section{Basic set-up}
\label{sec:Basic set-up}

\subsection{Nonprobability sample and probability sample}
\label{sec:Notation: two samples}

Let $\mathcal{U}=\{1,2,...,N\}$ 
be the index set of units for the finite population with size N. For each unit $i$ (where $i=1,2,...,N$), there are corresponding values of the $r$-dimensional vector of the auxiliary variable $x_i$ and the scalar response variable $y_i$, i.e., $x_i \in \mathbb{R}^r$ and $y_i \in \mathbb{R}$. 
Without loss of generality, we assume that the domain of $x_i$ is taken to be $[0,1]^r$. 
Under the design-based framework, the realized finite population values 
$\mathcal{F}_N=\{(x_i,y_i): i \in \mathcal{U}\}$ are treated as fixed. For the asymptotic analysis of the DNN estimator, we additionally view this finite population as generated from an underlying superpopulation distribution. The parameter of interest is the finite population mean
\[
\mu_y = N^{-1}\sum_{i=1}^N y_i
\]
of the response variable.

Suppose that a nonprobability sample $S_A$ of size $n_A$ is selected from $\mathcal{U}$ by a self-selection mechanism. Let $\{(x_i,y_i): i \in S_A \}$ denote the observed dataset from the nonprobability sample. 
Define the sample membership indicator for unit $i$ as $R_i = I(i \in S_A)$; that is, $R_i = 1$ if $i \in S_A$ and $R_i = 0$ otherwise, for $i=1,\ldots, N$. The underlying participation probability for unit $i$ in the nonprobability sample is defined as
\[
\pi_i^A = E_q(R_i\mid x_i,y_i)=P_q(R_i=1\mid x_i,y_i), \quad i=1,2,\ldots,N,
\]
where the subscript $q$ refers to the propensity score model for the selection mechanism of $S_A$. The corresponding implicit nonprobability sample weight is $w_i = 1/\pi_i^{A}$ for $i \in \mathcal{U}$.

The probability sample $S_B$ is assumed to be independently drawn from the same finite population under a known probability sampling design. Let $\{(x_i,d_i^B): i \in S_B\}$ denote the observed dataset from the probability sample, where $x_i$ is available from an existing survey, $d_i^B=1/\pi_i^B$ are the design weights, and $\pi_i^B=P(i \in S_B)$ are the first-order inclusion probabilities under the sampling design for $S_B$. Note that the response variable $y_i$ is not observed in the reference probability sample $S_B$. Define the combined sample as $S=S_A\cup S_B$, with $n=n_A+n_B$. It is assumed that there are no overlapping units between the two samples. The observed and unobserved components of the two samples are summarized in Table \ref{tab:tsd}.

\begin{table}[h]
	\caption{Two sources of data}
	\label{tab:tsd}
	\centering
	\begin{tabular}{l c c c }
		\hline
		\textbf{Sample} & \textbf{Sampling weight $\pi^{-1}$} & \textbf{Covariate $X$} & \textbf{Study variable $Y$} \\
		\hline
		Nonprobability sample & ? & $\surd$ & $\surd$ \\
		$\vdots$ & $\vdots$ & $\vdots$ & $\vdots$ \\
		$n_A $ & ? & $\surd$ & $\surd$ \\
		\hline
		Probability sample & $\surd$ & $\surd$ & ? \\
		$\vdots$ & $\vdots$ & $\vdots$ & $\vdots$ \\
		$n_A + n_B$ & $\surd$ & $\surd$ & ? \\
		\hline
	\end{tabular}
	\begin{flushleft}
		\footnotesize{$^1$Sample A denotes the nonprobability sample and sample B denotes the probability sample. The symbols `$\surd$' and `?' indicate observed and unobserved quantities, respectively.}
	\end{flushleft}
\end{table}

\subsection{Commonly used estimators for the finite population mean} 
\label{sec:Existing estimators}

Estimation of the population mean using the two samples generally requires modeling either the sampling score function $\pi_i^A$ for the nonprobability sample, the outcome mean function $m(x)$, or both. In practice, these functions are typically unknown and must be estimated. To address this issue, many researchers have modelled $\pi_i^A$ and $m(x)$ via parametric specifications $\pi_i^A(x^\top\theta)$ and $m(x^\top\beta)$, where $\theta$ and $\beta$ denote vectors of unknown model parameters. Several estimators of $\mu_y$ have been developed under different modeling assumptions. The following discussion examines three representative approaches and summarizes their respective theoretical properties and practical limitations within the context of survey sampling. Following \cite{aut2020}, we define three estimators of the population parameter $\mu_y$ as follows.

\noindent \textbf{Outcome regression estimator}
\begin{equation}
	\label{eq:Ore}
	\hat{\mu}_\mathrm{REG}=\frac{1}{\hat{N}^{B}}\sum_{i\in\mathcal{S}_B}d^{B}_i\hat{y}_i,
\end{equation}
where $d_i^B=1/\pi_i^B$ denotes the design weight for unit $i$ in the probability sample $S_B$, $\hat{N}^B=\sum_{i \in S_B} d_i^B$ is the estimated population size based on the probability sample, and $\hat{y}_i$ is the predicted value of the study variable for unit $i$, obtained from a fitted regression model.

Suppose that the finite population $\{(x_i,y_i): i \in \mathcal{U}\}$ can be regarded as a random sample drawn from the underlying model
\begin{equation*}
	y_i = m(x_i)+\varepsilon_i, \quad i=1,\ldots,N,
\end{equation*}
where $m(x_i)= E(y_i\mid x_i)$, which can take a parametric form such as $m(x_i)= x_i^\top\beta$ or a specified nonlinear parametric form. The error terms $\varepsilon_i$ are independent with $E(\varepsilon_i)=0$ and $V(\varepsilon_i)=v(x_i)\sigma^2$. The dataset from the nonprobability sample can be used to fit the outcome model. For the linear regression model $m(x_i)=x_i^\top\beta$ and the homogeneous variance structure $v(x_i)=1$, the least squares estimator of $\beta$ is given by
\begin{equation*}
	\hat{\beta}=\left(\sum_{i\in S_A}x_ix_i^\top\right)^{-1}\left(\sum_{i \in S_A}x_iy_i\right).
\end{equation*}
For a unit $i$ in the reference probability sample $S_B$, the predicted value of $y_i$ is then given by $\hat{y}_i = x_i^\top\hat{\beta}$. The validity of the regression estimator $\hat{\mu}_\mathrm{REG}$ relies on a correct specification of $m(x^\top\beta)$ and the consistency of $\hat{\beta}$. If $m(x^\top\beta)$ is misspecified or $\hat{\beta}$ is inconsistent, $\hat{\mu}_\mathrm{REG}$
can be biased. 

\noindent  \textbf{Inverse probability of sampling score weighting estimator}
\begin{equation}
	\label{eq:Ipsswe}
	\hat{\mu}_\mathrm{IPW}=\frac{1}{\hat{N}^{A}}\sum_{i\in\mathcal{S}_A}\frac{1}{\hat{\pi}^A_i}y_i,
\end{equation}
where $\hat{N}^A=\sum_{i \in S_A}1/\hat{\pi}_i^A$, $y_i$ is the observed value of the study variable in the nonprobability sample, and $\hat{\pi}_i^A$ is the estimated sampling score. When a parametric propensity score model is used, $\hat{\pi}_i^A$ is commonly obtained from a fitted logistic regression model. In this paper, we replace this parametric specification with the DNN-based sampling-score estimator described in \autoref{sec:Estimation of propensity scores by DNN}.

\noindent  \textbf{Doubly robust estimator}
\begin{equation}
	\label{eq:Dre}
	\hat{\mu}_\mathrm{DR}=\frac{1}{\hat{N}^{A}}\sum_{i\in\mathcal{S}_A}\frac{1}{\hat{\pi}_i^A}\{y_i-m(x_i,\hat{\beta})\}+\frac{1}{\hat{N}^{B}}\sum_{i\in\mathcal{S}_B}\frac{1}{\pi_i^B} m(x_i,\hat{\beta}).
\end{equation}
The doubly robust estimator $\hat{\mu}_\mathrm{DR}$ possesses an important theoretical property: under fixed-dimensional covariates $X$, it remains consistent as long as either the propensity score model $\pi_i^A(x^\top\theta)$ or the outcome model $m(x^\top\beta)$ is correctly specified, without requiring both to hold simultaneously. This double robustness property enhances the estimator's reliability in practical applications where model misspecification may occur.
%%%%%%%%%%%%%%%%%%
\section{Methodology and asymptotic property}
\label{sec:Methodology and asymptotic property}

In this section, we develop a DNN-based procedure for estimating the sampling scores of units in the nonprobability sample and then construct the corresponding inverse-probability weighted and doubly robust estimators for the finite population mean. We first consider the hypothetical case where $x_i$ is observed for all units in the finite population $\mathcal{U}$, whereas $y_i$ is observed only for units in the nonprobability sample $S_A$. The logit transformation of the participation probability $\pi_i^A$ is modeled as an unknown nonparametric function of the auxiliary variables, rather than being restricted to a prespecified linear form. We approximate this unknown function using a deep neural network (DNN), whose parameters are optimized by the ADAM algorithm. The resulting DNN-estimated sampling scores are then incorporated into the inverse-probability weighting estimator and the doubly robust framework, leading to the DIPW and DDR estimators, respectively.

\subsection{Estimation of propensity scores by DNN}
\label{sec:Estimation of propensity scores by DNN}

We assume a nonparametric model for the participation probability $\pi_i^A=\pi^A(x_i)$, specified through the logit link as
\begin{equation}
	\label{eq:Lf}
	\log\left(\frac{\pi_i^A}{1-\pi_i^A}\right)=g_0(x_i), \quad  i \in \mathcal{U},
\end{equation}
where $g_0(x_i)$ is an unknown real-valued nonparametric function, and $x_i$ is the vector of covariates associated with the $i$th unit in the finite population. If $x_i$ were observed for all units in $\mathcal{U}$, the population likelihood function for the participation indicators could be written as
\begin{equation}
	\label{eq:Lfdnn}
	L(g)=\prod_{i=1}^N \{\pi_i^A\}^{R_i}\{1-\pi_i^A\}^{1-R_i}.
\end{equation}
The corresponding maximum likelihood estimator of $\pi_i^A$ is given by
\[
\hat{\pi}_i^A=\frac{1}{1+\exp\{-\hat{g}(x_i)\}},
\]
where $\hat{g}$ is an estimator of $g_0$. The population log-likelihood is
\begin{equation}
	\label{eq:Llf}
	\begin{aligned}       
	 \ell(g)
	 &=\sum_{i=1}^N\big[R_i\log\pi_i^A + (1-R_i)\log(1-\pi_i^A)\big] \\
	 &=\sum_{i \in S_A}\log\frac{\pi_i^A}{1-\pi_i^A} + \sum_{i=1}^N\log(1-\pi_i^A).
	\end{aligned}
\end{equation}
However, the log-likelihood function in \eqref{eq:Llf} cannot be directly used in practice because the auxiliary variables $x_i$ are not available for all units in the finite population, and the population size $N$ may also be unknown. To address this issue, we estimate $g_0$ by maximizing the following pseudo-log-likelihood function:
\begin{equation}
	\label{eq:Pllf}
	\ell^*(g)=\sum_{i \in S_A}\log\frac{\pi_i^A}{1-\pi_i^A} + \sum_{i \in S_B} d_i^B\log(1-\pi_i^A), 
\end{equation}
where the population total $\sum_{i=1}^N\log(1-\pi_i^A)$ in $\ell(g)$ is replaced by the Horvitz--Thompson estimator $\sum_{i \in S_B} d_i^B\log(1-\pi_i^A)$ based on the reference probability sample $S_B$.

Under the nonparametric model \eqref{eq:Lf}, the pseudo-log-likelihood function \eqref{eq:Pllf} can be written as  
\begin{equation}
	\label{eq:Plldnnf}
	\ell^*(g)=\sum_{i \in S_A}g(x_i) - \sum_{i \in S_B} d_i^B\log\{1+\exp(g(x_i))\}.
\end{equation}
We define the loss function as $-\ell^*(g)$. In the sequel, a DNN is used to approximate $g_0$, and the DNN estimator $\hat{g}$ is obtained by applying the ADAM optimization algorithm to maximize the pseudo-log-likelihood function, or equivalently to minimize $-\ell^*(g)$. Once $\hat{g}$ is obtained, the estimated participation probability is
\[
\hat{\pi}_i^A=\frac{1}{1+\exp\{-\hat{g}(x_i)\}}.
\]
Using the DNN-estimated sampling scores, we define the DNN-assisted inverse-probability weighted estimator, denoted by $\hat{\mu}_\mathrm{DIPW}$, as
\begin{equation}
	\label{eq:Ipsswdnne}
	\hat{\mu}_\mathrm{DIPW}=\frac{1}{\hat{N}^{A}}\sum_{i\in S_A}\frac{1}{\hat{\pi}^A_i}y_i,
\end{equation}
where $\hat{N}^{A}=\sum_{i\in S_A}1/\hat{\pi}_i^A$.

We now briefly introduce the DNN class used to approximate $g_0$, following the notation of \cite{aut2022}; for further details on DNNs, see \cite{aut2016a}. Let $\mathbb{N}_{+}$ be the set of all positive integers. Given $K \in \mathbb{N}_{+}$ and $\mathbf{p}=(p_{0}, \ldots, p_{K}, p_{K+1})^{\top} \in \mathbb{N}_{+}^{K+2}$, a $(K+1)$-layer DNN with layer-width vector $\mathbf{p}$ is a composite function $g: \mathbb{R}^{p_{0}} \to \mathbb{R}^{p_{K+1}}$ recursively defined as  
\begin{equation}
	\label{eq:Gxf}
	\begin{aligned} 
		g(x) &= W_{K} g_{K}(x) + v_{K}, \\
		g_{K}(x) &= \sigma\left(W_{K-1} g_{K-1}(x) + v_{K-1}\right), \ldots, \\
		g_{1}(x) &= \sigma\left(W_{0} x + v_{0}\right).
	\end{aligned}   
\end{equation}
Here, $K$ denotes the depth of the network, and $\mathbf{p}$ lists the width of each layer. Specifically, $p_0$ is the dimension of the input variable, $p_1,\ldots,p_K$ are the dimensions of the $K$ hidden layers, and $p_{K+1}$ is the dimension of the output layer. The matrices $W_{k} \in \mathbb{R}^{p_{k+1} \times p_{k}}$ and vectors $v_{k} \in \mathbb{R}^{p_{k+1}}$, for $k = 0, \ldots, K$, are the parameters of the DNN. The matrix entry $(W_{k})_{ij}$ is the weight linking the $j$th neuron in layer $k$ to the $i$th neuron in layer $k+1$, and the vector entry $(v_{k})_{i}$ is the shift term associated with the $i$th neuron in layer $k+1$. 

The function $\sigma$ is a deterministic activation function applied componentwise to vectors, that is,
\[
\sigma\left[(x_{1}, \ldots, x_{p_k})^{\top}\right]
=
\left(\sigma(x_{1}), \ldots, \sigma(x_{p_k})\right)^{\top}.
\]
Thus, \(g_{k} = (g_{k1}, \ldots, g_{k p_{k}})^{\top}: \mathbb{R}^{p_{k-1}} \to \mathbb{R}^{p_{k}}\), for \(k = 1, \ldots, K\). While many activation functions are available, we focus on the rectified linear unit (ReLU), introduced in \cite{aut2010b}, defined by $\sigma(x)=\max\{x,0\}$.

Following \cite{aut2022}, we consider the following class of DNNs:
\begin{equation}
	\label{eq:Classdnn}
	\begin{aligned}
		\mathcal{G}(K, \mathbf{p}) = \biggl\{ g: \, & g \text{ is a DNN with } (K+1) \text{ layers and width vector } \mathbf{p}, \\
		& \max\left\{\|W_{k}\|_{\infty}, \|v_{k}\|_{\infty}\right\} \leq 1, \text{ for all } k = 0, \ldots, K \biggr\},
	\end{aligned}   
\end{equation}
where $\|\cdot\|_{\infty}$ denotes the sup-norm of a matrix or vector. For $s \in \mathbb{N}_{+}$ and $D>0$, we further define the class of sparse neural networks as
\begin{equation}
	\label{eq:Classsnn}
	\mathcal{G}(K, s, \mathbf{p}, D) := \left\{ g \in \mathcal{G}(K, \mathbf{p}): \sum_{k=0}^{K} \{\|W_{k}\|_{0} + \|v_{k}\|_{0}\} \leq s, \, \|g\|_{\infty} \leq D \right\},   
\end{equation}
where $\|\cdot\|_{0}$ denotes the number of nonzero entries of a matrix or vector, and $\|g\|_{\infty}$ is the sup-norm of the function $g$. For a sufficiently large fixed constant $D>0$, let $\mathcal{G}_D=\mathcal{G}(K,s,\mathbf{p},D)$. We estimate $g_0$ by
\begin{equation}
	\label{eq:mle}
	\hat{g}=\arg \max_{g \in \mathcal{G}_D } \ell^*(g)
	= \arg \min_{g \in \mathcal{G}_D } \left\{ -\ell^*(g)\right\}.
\end{equation}

To compute $\hat{g}$, we use the ADAM optimization algorithm to minimize the loss function $-\ell^*(g)$. The trainable parameters of the DNN, consisting of all weight matrices and bias vectors, are denoted by $\Theta$. The ADAM algorithm is a variant of stochastic gradient descent that iteratively updates the parameter estimates by adaptively estimating the first and second moments of the stochastic gradients of the empirical loss \citep{aut2014a}. It has been widely used in DNN-based statistical models, including the penalized deep partially linear Cox model studied by \cite{aut2024}.

In our setting, let
\[
f(\Theta)= -\ell^*\{g(\cdot \mid \Theta)\}.
\]
The ADAM algorithm is summarized in Algorithm 1, where the square, division, and square root operations are applied elementwise. For initialization, the bias terms are set to zero, and the weights are initialized using the Xavier initialization scheme, which helps stabilize gradient propagation during training \citep{aut2010b}. To improve numerical stability, a small constant $\epsilon_0>0$ is added to the denominator. Excessive iterations may lead to overfitting, so early stopping is used as a regularization strategy. Early stopping can mitigate overfitting and has been shown to support consistency of trained neural networks under suitable conditions \citep{aut2021a}.

\begin{algorithm*}
	\centering
	\begin{tabular}{@{}r l@{}}
		\toprule
		\multicolumn{2}{l}{\textbf{Algorithm 1: ADAM Algorithm}} \\
		\midrule
		\multicolumn{2}{l}{\textbf{Input:} Initial parameters $\Theta^{(0)}$, learning rate $\gamma$, 
			decay rates $r_1,r_2$, small constant $\epsilon_0$, threshold $l$} \\
		1  & Initialize $m^{(0)} \leftarrow 0$, $v^{(0)} \leftarrow 0$, $t \leftarrow 0$ \\
		2  & repeat \\
		3  & {\quad\vrule\ } Compute gradient $G^{(t)} \leftarrow \nabla_{\Theta} f(\Theta^{(t)})$ ; \\
		4  & {\quad\vrule\ } Update first moment: $m^{(t+1)} \leftarrow r_1 m^{(t)} + (1-r_1)G^{(t)}$ ; \\
		5  & {\quad\vrule\ } Update second moment: $v^{(t+1)} \leftarrow r_2 v^{(t)} + (1-r_2)\{G^{(t)}\}^2$ ; \\
		6  & {\quad\vrule\ } Bias correction: $\hat{m}^{(t+1)}=\dfrac{m^{(t+1)}}{1-r_1^{t+1}}$ ; \\
		7  & {\quad\vrule\ } Bias correction: $\hat{v}^{(t+1)}=\dfrac{v^{(t+1)}}{1-r_2^{t+1}}$ ; \\
		8  & {\quad\vrule\ } Update parameters: 
		$\Theta^{(t+1)} \leftarrow \Theta^{(t)}
		- \gamma \dfrac{\hat{m}^{(t+1)}}{\sqrt{\hat{v}^{(t+1)}}+\epsilon_0}$ ; \\
		9  & {\quad\vrule\ } $t \leftarrow t+1$ ; \\
		10 & until $\|\Theta^{(t)} - \Theta^{(t-1)}\|_2 \leq l$ \\
		\multicolumn{2}{l}{\textbf{Output:} Optimized parameters $\hat{\Theta}=\Theta^{(t)}$ and
		                         $\hat{g}=g(\cdot\mid \hat{\Theta})$.} \\
		\bottomrule
	\end{tabular}
\end{algorithm*}

\subsection{Deep doubly robust estimation by DNN}
\label{sec:Doubly robust estimation by DNN}

Using a fitted outcome regression model to predict the study variable for units in the probability sample and the DNN-estimated sampling scores for units in the nonprobability sample, we construct a deep doubly robust estimator for the finite population mean. The resulting estimator, denoted by $\hat{\mu}_\mathrm{DDR}$, is defined as
\begin{equation}
	\label{eq:Drednn}
	\hat{\mu}_\mathrm{DDR}
	=
	\frac{1}{\hat{N}^{A}}\sum_{i\in S_A}\frac{1}{\hat{\pi}_i^A}\{y_i-m(\bm{x}_i,\hat{\bm{\beta}})\}
	+
	\frac{1}{\hat{N}^{B}}\sum_{i\in S_B}\frac{1}{\pi_i^B} m(\bm{x}_i,\hat{\bm{\beta}}),
\end{equation}
where $\hat{N}^{A}=\sum_{i\in S_A}1/\hat{\pi}_i^A$ and $\hat{N}^{B}=\sum_{i\in S_B}1/\pi_i^B$. The first term uses the DNN-estimated sampling scores to weight the residuals in the nonprobability sample, whereas the second term uses the reference probability sample to estimate the population mean of the fitted outcome regression function.

\subsection{Asymptotic property}
\label{sec:Asymptotic property}

In this subsection, we establish the asymptotic properties of the DNN estimator $\hat{g}$ and the resulting estimators $\hat{\mu}_\mathrm{DIPW}$ and $\hat{\mu}_\mathrm{DDR}$ under regularity conditions. Detailed proofs are provided in \autoref{sec:Appendix}. 

We first introduce notation used in the theoretical results. For any vector 
$v=(v_{1}, \ldots, v_{p})^{\top} \in \mathbb{R}^{p}$, let
\[
\|v\|=\left(\sum_{i=1}^{p} v_i^2\right)^{1/2}, \qquad
\|v\|_{\infty}=\max_i |v_i|.
\]
For any matrix $W=(w_{ij}) \in \mathbb{R}^{m \times n}$, let $\|W\|_{\infty}=\max_{i,j}|w_{ij}|$. For any function $h$, let $\|h\|_{\infty}=\sup_x |h(x)|$ and let $\|h\|_{L^2}$ denote its $L^2$ norm. In particular,
\[
\|\hat{g}-g_0\|_{L^2}
=
\left[E\{\hat{g}(X)-g_0(X)\}^2\right]^{1/2}.
\]
Denote $a_n \lesssim b_n$ if $a_n \leq c b_n$ for some constant $c>0$, and $a_n \gtrsim b_n$ if $a_n \geq c b_n$ for some constant $c>0$. We write $a_n \asymp b_n$ if both $a_n \lesssim b_n$ and $a_n \gtrsim b_n$ hold. Furthermore, $\mathbb{P}_n$ and $\mathbb{P}$ denote the empirical measure and the probability measure, respectively. For example, for a function $f$, we write
\[
\mathbb{P}_{n} f=\frac{1}{n}\sum_{i=1}^{n}f(X_i), 
\qquad
\mathbb{P} f=\int f\,d\mathbb{P}.
\]

Some restrictions on the nonparametric function $g_0$ are required. We assume that $g_0$ belongs to a  H\"older class of smooth functions, as in \cite{aut2020b} and \cite{aut2022}. Specifically, for smoothness parameter $\gamma>0$, radius $M>0$, and domain $\mathbb{D}\subset\mathbb{R}^{r}$, define
\begin{equation*}
		\mathcal{H}_{r}^{\gamma}(\mathbb{D}, M)
		=
		\left\{
		g\colon\mathbb{D}\to\mathbb{R}\text{ such that }
		\sum_{\alpha:|\alpha|<\gamma}\left\|\partial^{\alpha} g\right\|_{\infty}
		+
		\sum_{\alpha:|\alpha|=\lfloor\gamma\rfloor}
		\sup_{x,y\in\mathbb{D},\,x\neq y}
		\frac{\left|\partial^{\alpha}g(x)-\partial^{\alpha}g(y)\right|}
		{\|x-y\|_{\infty}^{\gamma-\lfloor\gamma\rfloor}}
		\leq M
		\right\},
\end{equation*}
where $\lfloor\gamma\rfloor$ is the largest integer strictly smaller than $\gamma$, $\partial^{\alpha}:=\partial^{\alpha_{1}}\cdots\partial^{\alpha_r}$ with $\alpha=(\alpha_1,\ldots,\alpha_r)^\top$, and $|\alpha|=\sum_{j=1}^r\alpha_j$.

Let $q\in\mathbb{N}_{+}$, $M>0$,
$\bm{\gamma}=(\gamma_0,\ldots,\gamma_q)^\top\in\mathbb{R}_{+}^{q+1}$,
$\mathbf{d}=(d_0,\ldots,d_{q+1})^\top\in\mathbb{N}_{+}^{q+2}$, and
$\tilde{\mathbf{d}}=(\tilde d_0,\ldots,\tilde d_q)^\top\in\mathbb{N}_{+}^{q+1}$ with $\tilde d_j\leq d_j$, $j=0,\ldots,q$. We assume that $g_0$ belongs to the following composite smoothness class:
\begin{equation}
	\label{eq:Csfclass}
	\begin{aligned}
		\mathcal{H}(q, \bm{\gamma}, \mathbf{d}, \tilde{\mathbf{d}}, M)
		:=
		\big\{ g = g_q \circ \cdots \circ g_0 \; \text{such that }&
		g_i=(g_{i1},\ldots,g_{i d_{i+1}})^\top, \\
		& g_{ij}\in \mathcal{H}_{\tilde d_i}^{\gamma_i}
		([a_i,b_i]^{\tilde d_i},M), \\
		& \text{for some } |a_i|,|b_i|\leq M
		\big\}.
	\end{aligned}
\end{equation}
Functions in this class are characterized by two types of dimension parameters: $\mathbf{d}$ and $\tilde{\mathbf{d}}$. The vector $\mathbf{d}$ gives the ambient dimensions of the component functions, whereas $\tilde{\mathbf{d}}$ represents the intrinsic dimensions governing the complexity of the composition structure.

Define
\[
\tilde{\gamma}_i
=
\gamma_i\prod_{k=i+1}^q(\gamma_k\wedge 1),
\qquad
\gamma_n
=
\max_{i=0,\ldots,q}
n^{-\tilde{\gamma}_i/(2\tilde{\gamma}_i+\tilde d_i)},
\]
where $a\wedge b=\min(a,b)$. For notational simplicity, we use $n=N$ in the following theoretical statements and their proofs.

\begin{theorem}[Consistency and convergence rate of $\hat{g}$]
	\label{thm:g-rate}
	Under assumptions \ref{ass:A1}--\ref{ass:A3} and \ref{ass:B1}--\ref{ass:B2}, and regularity conditions \ref{cond:C1} and \ref{cond:C8}, we have
	\[
	\|\hat{g} - g_0\|_{L^2([0,1]^r)}
	=
	O_p(\gamma_n\log^2 n).
	\]
\end{theorem}

\begin{theorem}[Consistency and convergence rate of $\hat{\mu}_\mathrm{DIPW}$ and $\hat{\mu}_\mathrm{DDR}$]
	\label{thm:Consistency and convergence rate of dnn}
	Under assumptions \ref{ass:A1}--\ref{ass:A3}, \ref{ass:B1}--\ref{ass:B2}, and regularity conditions \ref{cond:C1}--\ref{cond:C9}, including correct specification of the nonparametric sampling-score function $g_0$, we have
	\[
	|\hat{\mu}_\mathrm{DIPW}-\mu_y|
	=
	O_p(\gamma_n\log^2 n),
	\qquad
	|\hat{\mu}_\mathrm{DDR}-\mu_y|
	=
	O_p(\gamma_n\log^2 n).
	\]
\end{theorem}

Under the composite smoothness condition $g_0\in\mathcal{H}(q,\bm{\gamma},\mathbf{d},\tilde{\mathbf{d}},M)$ in \eqref{eq:Csfclass}, \autoref{thm:g-rate} shows that the convergence rate of $\hat{g}$ is determined by the smoothness vector $\bm{\gamma}$ and the intrinsic dimension vector $\tilde{\mathbf{d}}$, rather than by the full ambient dimension alone. Consequently, the proposed DNN-based estimators can alleviate the curse of dimensionality and achieve faster convergence than conventional nonparametric estimators when the intrinsic dimension of the target function is relatively low.

%%%%%%%%%%%%%%%%%%
\section{Simulation studies}
\label{sec:Simulation studies}

To evaluate the proposed DNN-based methods for estimating a finite population mean, we adopt a simulation design motivated by \cite{aut2020}. 
We consider a finite population of size $N=20000$, in which the response variable $y$ and auxiliary variables $x$ follow the regression model
\begin{equation*}
	y_{i}=2 + x_{1i} + x_{2i} + x_{3i} + x_{4i} + \sigma \varepsilon_{i}, \quad i = 1,\ldots,N,  
\end{equation*}
where $x_{1i} = z_{1i}$, $x_{2i} = z_{2i} + 0.3x_{1i}$, $x_{3i} = z_{3i} + 0.2(x_{1i} + x_{2i})$, and $x_{4i} = z_{4i} + 0.1(x_{1i} + x_{2i} + x_{3i})$, with $z_{1i} \sim \mathrm{Bernoulli}(0.5)$, $z_{2i} \sim \mathrm{Uniform}(0,2)$, $z_{3i} \sim \mathrm{Exponential}(1)$, and $z_{4i} \sim\chi^{2}(4)$. The error terms $\varepsilon_{i}$ are independent and identically distributed as $N(0,1)$. The value of $\sigma$ is chosen by controlling the correlation coefficient $\rho$ between $y$ and $x^{\top}\beta$ at $0.3$, $0.5$, and $0.8$. 
The true propensity score $\pi_{i}^{A}$ for the nonprobability sample $S_{A}$ follows the logistic model
\begin{align*}
	\log\left(\frac{\pi_{i}^{A}}{1 - \pi_{i}^{A}}\right)=\ &\theta_{0} + 0.05x_{1i}x_{2i} + 0.1x_{2i}^2 + 0.05x_{3i}x_{4i} \\
	&+ 0.08\sin(0.3x_{3i}) + 0.05\ln(1+x_{2i}+x_{4i}),  
\end{align*}
which differs from the parametric linear propensity-score specification considered in \cite{aut2020}, since it is a nonlinear function of $x_i=(x_{1i},x_{2i},x_{3i},x_{4i})^\top$. Here, $\theta_{0}$ is chosen such that $\sum_{i = 1}^{N}\pi_{i}^{A} = n_{A}$ for a given target sample size $n_{A}$. 

We consider two scenarios for model specification. 
(i) The outcome regression model is correctly specified but the parametric propensity score model is misspecified; this scenario is denoted by \enquote{TF}. The fitted logistic regression model for the propensity score is misspecified as a linear model of the auxiliary variables,
\[
\log\{\pi_{i}^{A}/(1 -\pi_{i}^{A})\}=\theta_{0} +\theta_1 x_{1i} + \theta_2 x_{2i} + \theta_3x_{3i} + \theta_4 x_{4i}.
\]
(ii) Both the outcome regression model and the parametric propensity score model are misspecified; this scenario is denoted by \enquote{FF}. The outcome regression model is misspecified as
\[
y_{i}=\beta_0 + \beta_1x_{1i} + \beta_2x_{2i} + \beta_3x_{3i},
\]
with $x_{4i}$ omitted from the model. The fitted logistic regression model for the propensity score is the same as in scenario (i). For a given estimator $\hat{\mu}$, its performance is evaluated by the relative bias and mean squared error, computed as 
\begin{equation*}
	\%RB=\frac{1}{B}\sum_{b=1}^B\frac{\hat{\mu}^{(b)}-\mu_y}{\mu_y} \times 100, \qquad  
	MSE=\frac{1}{B}\sum_{b=1}^B(\hat{\mu}^{(b)}-\mu_y)^2, 
\end{equation*}
where $B$ is the number of simulation runs. In this study, $B=500$. 

\begin{table}[h]
	\caption{Simulated \%RB and MSE of the estimators of $\mu_y$ ($n_A = 500$, $n_B = 1000$)}
	\label{tab:sr}
	\centering
	\begin{tabular}{l l ccc ccc ccc }
		\hline
		\multirow{2}{*}{Models} & \multirow{2}{*}{Estimator} & \multicolumn{3}{c }{$\rho = 0.30$} & \multicolumn{3}{c }{$\rho = 0.50$} & \multicolumn{3}{c }{$\rho = 0.80$} \\
		\cline{3-11}
		& & \%RB & MSE & & \%RB & MSE & & \%RB & MSE & \\
		\hline
		\multirow{6}{*}{TF} 
		& $\hat{\mu}_A$ & 72.89 & 46.31 & & 74.67 & 47.97 & & 75.89 & 49.15 & \\
		& $\hat{\mu}_{\text{REG}}$ & -1.21 & 0.81 & & -0.67 & 0.25 & & -0.32 & 0.06 & \\
		& $\hat{\mu}_{\text{IPW}}$ & -5.77 & 18.35 & & -5.87 & 7.47 & & -5.92 & 3.80 & \\
		& $\hat{\mu}_{\text{DR}}$ & 0.97 & 13.82 & & 0.73 & 4.06 & & 0.26 & 0.81 & \\
		& $\hat{\mu}_{\text{DIPW}}$ & 1.24 & 2.90 & & 2.10 & 1.70 & & 2.76 & 1.37 & \\
		& $\hat{\mu}_{\text{DDR}}$ & -0.90 & 2.27 & & -0.46 & 0.62 & & -0.22 & 0.15 & \\
		\hline
		\multirow{6}{*}{FF} 
		& $\hat{\mu}_A$ & 72.86 & 46.28 & & 74.66 & 47.97 & & 75.89 & 49.15 & \\
		& $\hat{\mu}_{\text{REG}}$ & 77.18 & 51.96 & & 79.06 & 53.81 & & 80.35 & 55.12 & \\
		& $\hat{\mu}_{\text{IPW}}$ & -6.32 & 17.79 & & -5.80 & 7.48 & & -5.92 & 3.80 & \\ 
		& $\hat{\mu}_{\text{DR}}$ & -23.75 & 18.82 & & -24.21 & 10.07 & & -24.84 & 7.05 & \\
		& $\hat{\mu}_{\text{DIPW}}$ & 2.00 & 3.00 & & 2.77 & 1.77 & & 2.88 & 1.31 & \\
		& $\hat{\mu}_{\text{DDR}}$ & -0.21 & 2.68 & & -0.26 & 1.49 & & 0.17 & 0.88 & \\
		\hline
	\end{tabular}
\end{table}

The simulation results for $n_A=500$ and $n_B=1000$ are reported in \autoref{tab:sr}. Major observations from \autoref{tab:sr} can be summarized as follows. 

First, under the correctly specified outcome regression model and misspecified parametric propensity score model (\enquote{TF}), the naive estimator $\hat{\mu}_A=\bar{y}$ exhibits substantial positive bias and high MSE across all values of $\rho$. The regression estimator $\hat{\mu}_{\mathrm{REG}}$ performs well under this setting, as expected, because the outcome regression model is correctly specified. In contrast, $\hat{\mu}_{\mathrm{IPW}}$ suffers from negative bias and relatively large MSE because it relies exclusively on the misspecified parametric propensity score model. The estimator $\hat{\mu}_{\mathrm{DR}}$ improves upon $\hat{\mu}_{\mathrm{IPW}}$, reflecting the benefit of incorporating the correctly specified outcome model. The proposed $\hat{\mu}_{\mathrm{DIPW}}$ estimator improves substantially over $\hat{\mu}_{\mathrm{IPW}}$, indicating that the DNN can better approximate the nonlinear propensity score function. The proposed $\hat{\mu}_{\mathrm{DDR}}$ estimator further combines the DNN-estimated sampling scores with the outcome regression model and achieves small relative bias and low MSE in all three correlation settings.

Second, under the misspecified outcome regression model and misspecified parametric propensity score model (\enquote{FF}), the performance of both $\hat{\mu}_{\mathrm{REG}}$ and $\hat{\mu}_{\mathrm{IPW}}$ deteriorates, with substantial bias and increased MSE. The conventional doubly robust estimator $\hat{\mu}_{\mathrm{DR}}$ also performs poorly in this setting, since the standard double robustness property no longer protects against bias when both working models are misspecified. The performance of $\hat{\mu}_{\mathrm{IPW}}$ remains similar across the TF and FF scenarios because the same misspecified parametric propensity score model is used in both cases. In contrast, $\hat{\mu}_{\mathrm{DIPW}}$ and $\hat{\mu}_{\mathrm{DDR}}$ continue to perform well. In particular, $\hat{\mu}_{\mathrm{DDR}}$ has the smallest relative bias in the FF scenario across all three values of $\rho$. This improvement can be attributed to the flexibility of the DNN in approximating the nonlinear participation mechanism, thereby reducing the impact of parametric propensity-score misspecification. These results should be interpreted as finite-sample evidence under the simulation design considered here, rather than as a general guarantee of consistency under arbitrary simultaneous misspecification of both nuisance models.

\autoref{fig:simulation-results} provides a visualization of the distributions of the estimation biases under the scenarios in \autoref{tab:sr} using boxplots. 
\begin{center}
         \includegraphics[scale=0.6]{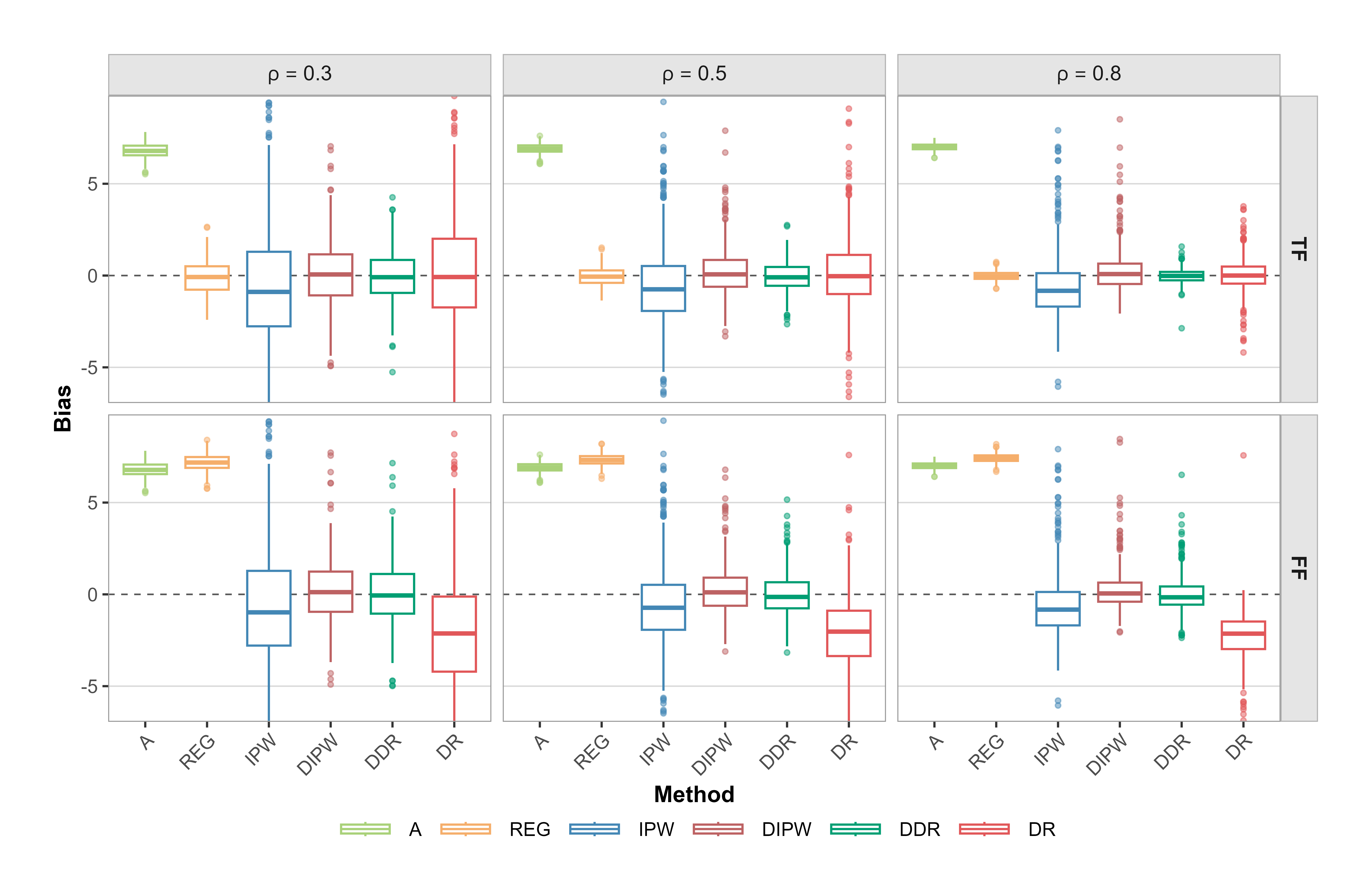} %{PLOT-2026May19-test.pdf}
	\captionof{figure}{Boxplots of the estimation biases under TF and FF scenarios when $\rho=0.3,0.5,0.8$}
	\label{fig:simulation-results} 
\end{center}

In summary, the proposed $\hat{\mu}_{\mathrm{DDR}}$ estimator exhibits favorable finite-sample performance under both the TF scenario, where the outcome model is correctly specified and the parametric propensity score model is misspecified, and the more challenging FF scenario, where both the outcome model and the parametric propensity score model are misspecified. Conventional parametric estimators, including $\hat{\mu}_{\mathrm{REG}}$, $\hat{\mu}_{\mathrm{IPW}}$, and $\hat{\mu}_{\mathrm{DR}}$, can suffer substantial performance loss under the FF scenario. By replacing the parametric propensity score model with a flexible DNN-based estimator, $\hat{\mu}_{\mathrm{DDR}}$ provides additional protection against propensity-score functional-form misspecification while retaining the structure of doubly robust estimation.

\section{Application}
\label{sec:Application}

In this section, we illustrate the proposed method using a dataset collected by the Pew Research Center (PRC) in 2015 (\href{http://www.pewresearch.org}{\textit{http://www.pewresearch.org}}). This dataset comprises nine nonprobability samples obtained from eight vendors, each using distinct and largely undocumented strategies for panel recruitment, sampling, participant incentives, and related procedures. For analytical purposes, we aggregate these samples into a single nonprobability sample with total sample size $n_A=9301$, hereafter referred to as the PRC dataset. To provide auxiliary population-level information, we use a reference probability sample from the 2015 Behavioral Risk Factor Surveillance System (BRFSS) survey (\href{https://www.cdc.gov/brfss/index.html}{\textit{https://www.cdc.gov/brfss/index.html}}), which contains 441,456 respondents.
These datasets were used by \cite{aut2020} to develop doubly robust estimators for integrating probability and nonprobability survey samples. They considered three datasets; here, we focus on the PRC and BRFSS datasets. \autoref{tab:dd} reports the marginal distributions of selected common variables from the three datasets, where $\hat{\mu}_{\text{PRC}}$ and $\hat{\mu}_{\text{BRFSS}}$ denote the corresponding survey-weighted estimates of population means.

\begin{table}[h]
	\caption{Marginal distributions of common covariates from the two samples}
	\label{tab:dd}
	\centering
	\begin{tabular}{llrr}
		\toprule
		& & $\hat{\mu}_{\text{PRC}}$ & $\hat{\mu}_{\text{BRFSS}}$ \\
		\midrule
		\multirow{4}{*}{Age category} & $<$30         & 18.29 & 20.91 \\
		& $\geq$30, $<$50 & 32.60 & 33.28 \\
		& $\geq$50, $<$70 & 38.68 & 32.69 \\
		& $\geq$70        & 10.43 & 13.12 \\
		\midrule
		Gender & Female                      & 54.36 & 51.32 \\
		Race   & White only                  & 82.28 & 75.05 \\
		Race   & Black only                  & 8.83  & 12.59 \\
		Origin & Hispanic/Latino            & 9.27  & 16.52 \\
		Region & Northeast                 & 19.96 & 17.70 \\
		Region & South                     & 33.26 & 38.27 \\
		Region & West                      & 24.09 & 23.18 \\
		Marital status & Married           & 50.35 & 50.82 \\
		Employment & Working         & 52.13 & 56.63 \\
		Employment & Retired           & 24.34 & 17.93 \\
		Education & High school or less    & 21.63 & 42.66 \\
		Education & Bachelor's degree and above & 41.57 & 26.32 \\
	\bottomrule
	\end{tabular}
\end{table}

To illustrate the proposed estimators and compare them with the estimators of \cite{aut2020}, including 
$\hat{\mu}_{\mathrm{REG}}$, $\hat{\mu}_{\mathrm{IPW}}$, and $\hat{\mu}_{\mathrm{DR}}$, we estimate the population means of seven outcome variables. The first six outcome variables are binary, while the last one is continuous. \autoref{tab:epm} presents the estimation results based on the set of covariates that are available in \autoref{tab:dd}. The covariates used in the analysis include age, gender, race, Hispanic origin, region, marital status, employment status, and education level, categorized as high school or less versus bachelor's degree and above. The variable \enquote{age} is coded as 1, 2, 3, and 4 for the four age categories and treated as a continuous variable in the analysis. The estimator $\hat{\mu}_{A}$ denotes the unadjusted simple sample mean from the nonprobability PRC sample. 

\autoref{tab:epm} shows that the adjusted estimators generally differ from the unadjusted nonprobability sample mean $\hat{\mu}_{A}$, reflecting the effect of using auxiliary information from the reference probability sample. The proposed DNN-based estimators $\hat{\mu}_{\mathrm{DIPW}}$ and $\hat{\mu}_{\mathrm{DDR}}$ produce estimates that are broadly comparable to those from $\hat{\mu}_{\mathrm{REG}}$, $\hat{\mu}_{\mathrm{IPW}}$, and $\hat{\mu}_{\mathrm{DR}}$ for several outcomes, although some differences are observed, particularly for the outcomes related to trust in neighbors, local election voting, and days with at least one drink last month. Unlike the simulation study, the true population means and the true model specifications are unknown in this empirical application. Therefore, the application should be interpreted as an illustration of implementation and comparison across estimators, rather than as a definitive assessment of estimator accuracy.

\begin{table}[h]
	\caption{Estimated population mean of $y$ using a set of common covariates}
	\label{tab:epm}
	\centering
	\begin{tabular}{p{4.5cm} r r r r r r}
		\toprule
		Response variable $y$ & $\hat{\mu}_A$ & $\hat{\mu}_{REG}$ & $\hat{\mu}_{IPW}$ & $\hat{\mu}_{DR}$ & $\hat{\mu}_{DIPW}$ & $\hat{\mu}_{DDR}$ \\
		\midrule
		\makecell[l]{\%Talked with \\ neighbors frequently} & 46.13 & 45.60 & 45.48 & 45.63 & 46.12 & 44.83\\
		\addlinespace[0.5pt] 
		\cmidrule(r){1-7} 
		\addlinespace[0.5pt]
		\makecell[l]{\%Tended to trust \\ neighbors} & 58.97 & 54.67 & 54.93 & 54.68 & 51.71 & 52.07 \\
		\addlinespace[0.5pt] 
		\cmidrule(r){1-7} 
		\addlinespace[0.5pt]
		\makecell[l]{\%Expressed opinions \\ at a government level} & 26.54 & 23.75 & 23.89 & 23.88 & 26.78 & 22.34 \\
		\addlinespace[0.5pt] 
		\cmidrule(r){1-7} 
		\addlinespace[0.5pt]
		\makecell[l]{\%Voted local  elections} & 74.97 & 70.02 & 70.32 & 70.11 & 70.98 & 70.05 \\
		\addlinespace[0.5pt] 
		\cmidrule(r){1-7} 
		\addlinespace[0.5pt]
		\makecell[l]{\%Participated in \\ school groups}  & 20.97 & 20.28 & 20.30 & 20.30 & 17.12 & 19.90 \\
		\addlinespace[0.5pt] 
		\cmidrule(r){1-7} 
		\addlinespace[0.5pt]
		\makecell[l]{\%Participated in \\ service organizations}  & 14.11 & 13.36 & 13.11 & 13.27 & 12.45 & 12.08 \\
		\addlinespace[0.5pt] 
		\cmidrule(r){1-7} 
		\addlinespace[0.5pt]
		\makecell[l]{Days had at least \\ one drink last month}  & 5.30  & 4.86 & 4.87  & 4.89 & 5.33  & 4.37 \\
		\bottomrule
	\end{tabular}
\end{table}

\section{Concluding remarks and discussion}
\label{sec:Conclusion}

In this article, we propose a deep doubly robust estimation framework for integrating probability and nonprobability survey samples. The proposed method uses deep neural networks to estimate the sampling propensity score, thereby replacing a restrictive parametric propensity score model with a flexible data-adaptive alternative. The DNN parameters are optimized using the ADAM algorithm, allowing the method to accommodate complex nonlinear selection mechanisms that may arise in nonprobability samples. The theoretical results establish consistency and convergence rates for the proposed estimators under regularity conditions. The simulation studies show that the DDR estimator achieves stable finite-sample performance, particularly when the parametric propensity score model is misspecified and the true selection mechanism is nonlinear. The empirical application illustrates how the proposed estimators can be implemented using a nonprobability survey sample together with auxiliary information from a reference probability sample.

The proposed framework provides a useful extension of doubly robust estimation for modern survey settings involving heterogeneous data sources. Its main advantage is that it reduces reliance on a correctly specified parametric propensity score model while preserving the structure of inverse-probability weighted and doubly robust estimation. At the same time, the method does not eliminate the need for standard identification assumptions, such as the availability of sufficiently rich auxiliary variables and positivity of participation probabilities. Moreover, the theoretical guarantees are established under regularity conditions for the DNN-based sampling-score estimator and should not be interpreted as unconditional robustness to arbitrary misspecification of all nuisance components.

In the present work, we use a parametric outcome regression model and a nonparametric DNN-based propensity score model. This choice reflects our focus on the sampling-score estimation problem, where the pseudo-likelihood involves both the nonprobability sample and the reference probability sample and therefore requires additional theoretical treatment. A natural extension is to also model the outcome regression function nonparametrically, for example by using another DNN. Such an extension may further improve performance when both the outcome model and the propensity score model contain complex nonlinear structures, although it would also increase computational cost and require additional theoretical analysis. Developing this fully nonparametric deep doubly robust framework is an important direction for future research.

%%%%%%%%%%%%%%%%%%%

\section*{Data availability}
%  <Insert the statement about the availability or absence of shared data.>
 The PRC data analyzed in this study are publicly available at 
 \href{http://www.pewresearch.org}{\textit{http://www.pewresearch.org}}.
 The BRFSS data are publicly available at  \href{https://www.cdc.gov/brfss/index.html}{\textit{https://www.cdc.gov/brfss/index.html}}. All data were accessed in accordance with the respective data-sharing policies of the providers.

\section*{Funding}
%  <Place your funding information here.>
 This research was supported by grants from the Natural Sciences and
Engineering Research Council of Canada.

%% Appendices
\appendix                       % start of appendices
\renewcommand{\thesubsection}{\thesection\arabic{subsection}}
\section{Appendix}
\label{sec:Appendix}

\renewcommand{\theequation}{A\arabic{equation}} 
\renewcommand{\theHequation}{A\arabic{equation}}
\setcounter{equation}{0} 

In this appendix, we first state the regularity conditions used in the technical arguments. We then provide three auxiliary lemmas and their proofs, followed by the proofs of the two theorems stated in the main text. The theoretical arguments combine a superpopulation representation for the DNN sampling-score model with a sequence of finite populations and probability samples. Specifically, the finite population is viewed as generated from an underlying superpopulation distribution, while conditional on the finite population, the reference probability sample $S_B$ is used to estimate population totals through the probability sampling design. This convention allows us to use empirical-process notation for the DNN estimator and design-consistency arguments for Horvitz--Thompson estimators based on $S_B$.

\subsection{Regularity conditions}

Let $f(y\mid x)$ be the conditional distribution of $y$ given $x$ in the superpopulation model that generates the finite outcome population. Let $m(x,\beta)$ be the mean function of the outcome regression model. Let $\beta_0$ denote the true value of the model parameter when the outcome regression model is correctly specified. When the outcome regression model is misspecified, let $\beta^*$ denote the probability limit of $\hat{\beta}$. The sampling-score function is modeled nonparametrically through $\pi_i^A=\sigma\{g_0(x_i)\}$, where $\sigma(z)=\{1+\exp(-z)\}^{-1}$ and $g_0$ is the true logit sampling-score function.

\begin{assumptionA}
	\label{ass:A1}
	The selection indicator $R_i$ and the response variable $y_i$ are independent given the set of covariates $x_i$, i.e.,
	\[
	\pi_i^A=P(R_i=1\mid x_i,y_i)=P(R_i=1\mid x_i),
	\]
	for all $i$. The quantity $\pi_i^A$ is referred to as the sampling score or participation probability.
\end{assumptionA}

\begin{assumptionA}
	\label{ass:A2}
	All units have a nonzero propensity score, that is, $\pi_i^A>0$ for all $i$.
\end{assumptionA}

\begin{assumptionA}
	\label{ass:A3}
	The indicator variables $R_i$ and $R_j$ are independent given $x_i$ and $x_j$ for $i\neq j$.
\end{assumptionA}

\begin{assumptionB}
	\label{ass:B1}
	The DNN architecture satisfies
	\[
	K=O(\log n), \qquad s=O(n\gamma_n^2\log n),
	\]
	and
	\[
	n\gamma_n^2 \lesssim \min_{k=1,\ldots,K}p_k
	\leq \max_{k=1,\ldots,K}p_k \lesssim n,
	\]
	where $K$, $s$, and $\mathbf{p}$ are defined in \autoref{eq:Classsnn}.
\end{assumptionB}

\begin{assumptionB}
	\label{ass:B2}
	The nonparametric logit sampling-score function $g_0$ is an element of the composite smoothness class
	$\mathcal{H}(q,\bm{\gamma},\mathbf{d},\tilde{\mathbf{d}},M)$.
\end{assumptionB}

\begin{conditionC}
	\label{cond:C1}
	The population size $N$ and the sample sizes $n_A$ and $n_B$ satisfy
	\[
	\lim_{N\to\infty}\frac{n_A}{N}=f_A\in(0,1),
	\qquad
	\lim_{N\to\infty}\frac{n_B}{N}=f_B\in(0,1).
	\]
\end{conditionC}

\begin{conditionC}
	\label{cond:C2}
	For each $x$, $\partial m(x,\beta)/\partial\beta$ is continuous in $\beta$ and
	\[
	\left\|\frac{\partial m(x,\beta)}{\partial\beta}\right\|\le h(x,\beta)
	\]
	for $\beta$ in a neighborhood of $\beta_0$, and
	\[
	N^{-1}\sum_{i=1}^N h(x_i,\beta_0)=O(1).
	\]
\end{conditionC}

\begin{conditionC}
	\label{cond:C3}
	For each $x$, $\partial^2m(x,\beta)/\partial\beta\partial\beta^\top$ is continuous in $\beta$ and
	\[
	\max_{j,l}\left|
	\frac{\partial^2m(x,\beta)}
	{\partial\beta_j\partial\beta_l}
	\right|
	\le k(x,\beta)
	\]
	for $\beta$ in a neighborhood of $\beta_0$, and
	\[
	N^{-1}\sum_{i=1}^N k(x_i,\beta_0)=O(1).
	\]
\end{conditionC}

\begin{conditionC}
	\label{cond:C4}
	The finite population and the sampling design for $S_B$ satisfy
	\[
	N^{-1}\sum_{i\in S_B}d_i^B u_i
	-
	N^{-1}\sum_{i=1}^N u_i
	=
	O_p(n_B^{-1/2}),
	\]
	for $u_i=x_i$, $y_i$, $m(x_i,\beta)$, and $\dot m(x_i,\beta)=\partial m(x_i,\beta)/\partial\beta$, whenever these quantities are used in the expansion.
\end{conditionC}

\begin{conditionC}
	\label{cond:C5}
	There exist constants $0<c_1<c_2<\infty$ such that
	\[
	0<c_1\leq \frac{N\pi_i^A}{n_A}\leq c_2,
	\qquad
	0<c_1\leq \frac{N\pi_i^B}{n_B}\leq c_2
	\]
	for all units $i$.
\end{conditionC}

\begin{conditionC}
	\label{cond:C6}
	The finite population and the propensity scores satisfy
	\[
	N^{-1}\sum_{i=1}^N y_i^2=O(1),
	\qquad
	N^{-1}\sum_{i=1}^N \|x_i\|^3=O(1),
	\]
	and
	\[
	N^{-1}\sum_{i=1}^N \pi_i^A(1-\pi_i^A)x_ix_i^\top
	\]
	is positive definite.
\end{conditionC}

\begin{conditionC}
	\label{cond:C7}
	The logistic link function $\sigma(z)=\{1+\exp(-z)\}^{-1}$ has globally bounded first and second derivatives. That is, there exist constants $L_\sigma>0$ and $M_\sigma>0$ such that
	\[
	|\sigma'(z)|\leq L_\sigma,
	\qquad
	|\sigma''(z)|\leq M_\sigma,
	\]
	for all $z\in\mathbb{R}$.
\end{conditionC}

\begin{conditionC}
	\label{cond:C8}
	For any fixed $D>0$, the probability sampling design for $S_B$ satisfies the following uniform design-consistency condition over the bounded sparse DNN class:
	\[
	\sup_{g\in\mathcal{G}(K,s,\mathbf{p},D)}
	\left|
	N^{-1}\sum_{i\in S_B}d_i^B\log\{1+\exp(g(x_i))\}
	-
	N^{-1}\sum_{i=1}^N\log\{1+\exp(g(x_i))\}
	\right|
	=
	O_p(N^{-1/2}).
	\]
	This condition strengthens Condition \ref{cond:C4} from fixed functions to the DNN-indexed class needed for the pseudo-likelihood argument.
\end{conditionC}

\begin{conditionC}
	\label{cond:C9}
	The estimated sampling scores used in the DIPW and DDR estimators are truncated away from zero and one. Specifically, for a sequence $\epsilon_n>0$ satisfying $\epsilon_n\to 0$ sufficiently slowly, define
	\[
	\hat{\pi}_{i,\epsilon}^A
	=
	\min\{1-\epsilon_n,\max(\epsilon_n,\hat{\pi}_i^A)\}.
	\]
	In the proofs, $\hat{\pi}_i^A$ denotes this truncated version.
\end{conditionC}

Assumptions \ref{ass:A1}--\ref{ass:A3} are standard assumptions for inference with nonprobability samples. Assumption \ref{ass:A1} is the ignorability assumption and corresponds to missing at random as defined by \cite{aut1976} and Little and \cite{aut2019}. Under Assumption \ref{ass:A1}, the outcome mean function satisfies $E(y\mid x)=E(y\mid x,R=1)$ and can therefore be estimated from the nonprobability sample. Assumption \ref{ass:A2} rules out zero participation probabilities. Assumptions \ref{ass:A1} and \ref{ass:A2} together form the strong ignorability condition of \cite{aut1983}. Assumption \ref{ass:A3} rules out residual dependence among sample membership indicators after conditioning on the covariates.

Assumptions \ref{ass:B1}--\ref{ass:B2} control the complexity of the DNN estimator. Assumption \ref{ass:B1} balances approximation error and estimation error through the depth, width, and sparsity level of the neural network. Assumption \ref{ass:B2} restricts the true logit sampling-score function $g_0$ to a composite H\"older class, under which sparse ReLU networks can approximate $g_0$ at the rate used in the main theorem.

Conditions \ref{cond:C1}--\ref{cond:C6} follow the standard two-sample asymptotic framework for integrating nonprobability and probability samples. Condition \ref{cond:C1} implies that $O_p(n_A^{-1/2})$, $O_p(n_B^{-1/2})$, and $O_p(N^{-1/2})$ have the same order. Conditions \ref{cond:C2} and \ref{cond:C3} are smoothness and boundedness conditions for Taylor expansions involving the outcome model. Condition \ref{cond:C4} is a design-consistency condition for fixed population quantities. Condition \ref{cond:C5} requires that the participation probabilities and probability-sample inclusion probabilities are of the same order as those under simple random sampling. Condition \ref{cond:C6} imposes finite moment and nonsingularity requirements. Condition \ref{cond:C7} is used to transfer the convergence rate of $\hat g$ to the estimated sampling scores $\hat{\pi}_i^A=\sigma\{\hat g(x_i)\}$. Condition \ref{cond:C8} is needed because the pseudo-log-likelihood replaces an unavailable population total by a Horvitz--Thompson estimator uniformly over a DNN-indexed class. Condition \ref{cond:C9} ensures that inverse estimated sampling scores are well behaved in the DIPW and DDR estimators.

\subsection{Lemma 1 and proof}

\begin{lemma}
	\label{lem:Lemma1}
	Define
	\[
	\mathcal{F}_1=\{Rg(x): g\in\mathcal{G}(K,s,\mathbf{p},D)\}
	\]
	and
	\[
	\mathcal{F}_2=
	\{\log(1+\exp\{g(x)\}): g\in\mathcal{G}(K,s,\mathbf{p},D)\}.
	\]
	Then, for any $D>0$, $\mathcal{F}_1$ and $\mathcal{F}_2$ are $P$-Glivenko--Cantelli.
\end{lemma}

\begin{proof}[Proof of Lemma 1]
For any $g_1,g_2\in\mathcal{G}(K,s,\mathbf{p},D)$, define $h(z)=\log(1+\exp z)$. Since
\[
h'(z)=\frac{\exp z}{1+\exp z}\in(0,1),
\]
the mean value theorem gives
\[
\left|h\{g_1(x)\}-h\{g_2(x)\}\right|
\leq
|g_1(x)-g_2(x)|.
\]
Therefore,
\[
\mathbb{P}\left|\log(1+\exp\{g_1(x)\})-\log(1+\exp\{g_2(x)\})\right|
\leq
\mathbb{P}|g_1(x)-g_2(x)|.
\]
This implies that
\[
\mathcal{N}(\epsilon,\mathcal{F}_2,L^1(P))
\leq
\mathcal{N}\{\epsilon,\mathcal{G}(K,s,\mathbf{p},D),L^1(P)\}.
\]
Similarly, since $R\in\{0,1\}$,
\[
\mathbb{P}|Rg_1(x)-Rg_2(x)|
\leq
\mathbb{P}|g_1(x)-g_2(x)|,
\]
and hence
\[
\mathcal{N}(\epsilon,\mathcal{F}_1,L^1(P))
\leq
\mathcal{N}\{\epsilon,\mathcal{G}(K,s,\mathbf{p},D),L^1(P)\}.
\]
The result follows from Theorem 19.13 of \cite{aut2000} and the entropy bound for sparse ReLU networks in Lemma 6 of \cite{aut2022}.
\end{proof}

\subsection{Lemma 2 and proof}

\begin{lemma}
	\label{lem:Lemma2}
	Let
	\[
	L_0(g)=E\left[Rg(X)-\log\{1+\exp(g(X))\}\right].
	\]
	Under the regularity conditions on $g$ and $g_0$, there exists a sufficiently small constant $c>0$ such that, for all $g$ satisfying $d(g,g_0)\leq c$,
	\[
	L_0(g)-L_0(g_0)\asymp -d^2(g,g_0),
	\]
	where $d^2(g,g_0)=\|g-g_0\|_{L^2}^2$.
\end{lemma}

\begin{proof}[Proof of Lemma 2]
We have
\[
\begin{aligned}
L_0(g)-L_0(g_0)
&=
E\left[R\{g(X)-g_0(X)\}\right] \\
&\quad
-
E\left[
\log\{1+\exp(g(X))\}
-
\log\{1+\exp(g_0(X))\}
\right].
\end{aligned}
\]
Let $f(z)=\log(1+\exp z)$. Then
\[
f'(z)=\frac{\exp z}{1+\exp z},
\qquad
f''(z)=\frac{\exp z}{(1+\exp z)^2}.
\]
By Taylor expansion, for some $\xi(X)$ between $g_0(X)$ and $g(X)$,
\[
\begin{aligned}
f\{g(X)\}-f\{g_0(X)\}
&=
\frac{\exp\{g_0(X)\}}{1+\exp\{g_0(X)\}}
\{g(X)-g_0(X)\} \\
&\quad+
\frac{1}{2}
\frac{\exp\{\xi(X)\}}{[1+\exp\{\xi(X)\}]^2}
\{g(X)-g_0(X)\}^2.
\end{aligned}
\]
Therefore,
\[
\begin{aligned}
L_0(g)-L_0(g_0)
&=
E\left[
\left\{R-\frac{\exp\{g_0(X)\}}{1+\exp\{g_0(X)\}}\right\}
\{g(X)-g_0(X)\}
\right] \\
&\quad-
\frac{1}{2}
E\left[
\frac{\exp\{\xi(X)\}}{[1+\exp\{\xi(X)\}]^2}
\{g(X)-g_0(X)\}^2
\right].
\end{aligned}
\]
The first term is zero because
\[
E(R\mid X)=\frac{\exp\{g_0(X)\}}{1+\exp\{g_0(X)\}}.
\]
Moreover, since $g\in\mathcal{G}(K,s,\mathbf{p},D)$ and $g_0$ is bounded under the stated conditions, there exist constants $0<a<b<\infty$ such that
\[
a
\leq
\frac{\exp\{\xi(X)\}}{[1+\exp\{\xi(X)\}]^2}
\leq
b.
\]
It follows that
\[
-\frac{b}{2}\|g-g_0\|_{L^2}^2
\leq
L_0(g)-L_0(g_0)
\leq
-\frac{a}{2}\|g-g_0\|_{L^2}^2,
\]
which proves the claim.
\end{proof}

\subsection{Lemma 3 and proof}

\begin{lemma}
	\label{lem:Lemma3}
	Let
	\[
	\mathcal{B}_{\delta}
	=
	\{g\in\mathcal{G}(K,s,\mathbf{p},D):\|g-g_0\|_{L^2}\leq\delta\}.
	\]
	Define $\mathbb{G}_n=\sqrt{n}(\mathbb{P}_n-\mathbb{P})$,
	\[
	I=\mathbb{G}_n(g-g_0),
	\]
	and
	\[
	II=
	\mathbb{G}_n
	\left[
	\log(1+\exp g)-\log(1+\exp g_0)
	\right].
	\]
	Then
	\[
	\mathbb{E}^*\sup_{g\in\mathcal{B}_{\delta}}|I|
	=
	O\left(
	\delta\sqrt{s\log\frac{U}{\delta}}
	+
	\frac{s}{\sqrt n}\log\frac{U}{\delta}
	\right),
	\]
	and
	\[
	\mathbb{E}^*\sup_{g\in\mathcal{B}_{\delta}}|II|
	=
	O\left(
	\delta\sqrt{s\log\frac{U}{\delta}}
	+
	\frac{s}{\sqrt n}\log\frac{U}{\delta}
	\right),
	\]
	where $\mathbb{E}^*$ denotes outer expectation and
	\[
	U=K\prod_{k=0}^K(p_k+1)\sum_{k=0}^Kp_kp_{k+1}.
	\]
\end{lemma}

\begin{proof}[Proof of Lemma 3]
Let
\[
\mathcal{F}_{1\delta}=\{g-g_0:g\in\mathcal{B}_{\delta}\},
\]
and
\[
\mathcal{F}_{2\delta}
=
\{\log(1+\exp g)-\log(1+\exp g_0):g\in\mathcal{B}_{\delta}\}.
\]
For $i=1,2$, write
\[
\|\mathbb{G}_n\|_{\mathcal{F}_{i\delta}}
=
\sup_{f\in\mathcal{F}_{i\delta}}|\mathbb{G}_n f|.
\]
For any $g,g_1\in\mathcal{B}_{\delta}$,
\[
E\{(g-g_1)^2\}\lesssim d^2(g,g_1).
\]
Moreover, since $z\mapsto\log(1+\exp z)$ is Lipschitz with constant one,
\[
E\left[
\{\log(1+\exp g)-\log(1+\exp g_1)\}^2
\right]
\leq
E\{g-g_1\}^2
\lesssim
d^2(g,g_1).
\]
Therefore, by the sparse-DNN entropy bound,
\[
\log\mathcal{N}_{[]}(\epsilon,\mathcal{F}_{i\delta},L^2(P))
\lesssim
s\log\left(\frac{U}{\epsilon}\right),
\qquad i=1,2,
\]
provided $p\leq s$ and $\delta\leq U$. Consequently,
\[
\begin{aligned}
J_{[]}(\delta,\mathcal{F}_{i\delta})
&:=
\int_0^\delta
\sqrt{
1+\log\mathcal{N}_{[]}(\epsilon,\mathcal{F}_{i\delta},L^2(P))
}
\,d\epsilon \\
&\lesssim
\int_0^\delta
\sqrt{
1+s\log\left(\frac{U}{\epsilon}\right)
}
\,d\epsilon \\
&\asymp
\delta\sqrt{s\log\left(\frac{U}{\delta}\right)}.
\end{aligned}
\]
Combining this bound with Lemma 3.4.2 of \cite{aut1996}, we obtain
\[
\mathbb{E}^*\|\mathbb{G}_n\|_{\mathcal{F}_{i\delta}}
\lesssim
J_{[]}(\delta,\mathcal{F}_{i\delta})
\left\{
1+
\frac{J_{[]}(\delta,\mathcal{F}_{i\delta})}{\delta^2\sqrt n}
\right\},
\]
which yields
\[
\mathbb{E}^*\|\mathbb{G}_n\|_{\mathcal{F}_{i\delta}}
=
O\left(
\delta\sqrt{s\log\frac{U}{\delta}}
+
\frac{s}{\sqrt n}\log\frac{U}{\delta}
\right).
\]
The stated bounds for $I$ and $II$ follow.
\end{proof}

\subsection{Proof of Theorem 1}

Let $\pi_i^A=\pi^A(x_i)$ satisfy
\[
\log\left(\frac{\pi_i^A}{1-\pi_i^A}\right)=g(x_i).
\]
For the theoretical argument, define the normalized population log-likelihood and pseudo-log-likelihood by
\[
L_N(g)=\frac{1}{N}
\left[
\sum_{i\in S_A}g(x_i)
-
\sum_{i=1}^N\log\{1+\exp(g(x_i))\}
\right],
\]
and
\[
L_N^*(g)=\frac{1}{N}
\left[
\sum_{i\in S_A}g(x_i)
-
\sum_{i\in S_B}d_i^B\log\{1+\exp(g(x_i))\}
\right].
\]
Let
\[
\ell(g)=Rg(X)-\log\{1+\exp(g(X))\},
\qquad
L_0(g)=E\{\ell(g)\}.
\]
The only difference between $L_N(g)$ and $L_N^*(g)$ is that the unavailable finite-population total
\[
\sum_{i=1}^N\log\{1+\exp(g(x_i))\}
\]
is replaced by its Horvitz--Thompson estimator based on the reference probability sample $S_B$. By Condition \ref{cond:C8}, for any fixed $D>0$,
\begin{equation}
	\label{eq:appendix-A1}
	\sup_{g\in\mathcal{G}(K,s,\mathbf{p},D)}
	|L_N(g)-L_N^*(g)|
	=
	O_p(N^{-1/2}).
\end{equation}

By Lemma \ref{lem:Lemma1},
\[
\mathcal{F}_1=\{Rg(X):g\in\mathcal{G}(K,s,\mathbf{p},D)\},
\qquad
\mathcal{F}_2=\{\log(1+\exp\{g(X)\}):g\in\mathcal{G}(K,s,\mathbf{p},D)\}
\]
are $P$-Glivenko--Cantelli. Hence,
\begin{equation}
	\label{eq:appendix-A2}
	\sup_{g\in\mathcal{G}(K,s,\mathbf{p},D)}
	|L_N(g)-L_0(g)|
	\stackrel{p}{\to}0.
\end{equation}
Let $L_0^*(g)=E\{L_N^*(g)\}$. By the unbiasedness of the Horvitz--Thompson replacement,
\begin{equation}
	\label{eq:appendix-A3}
	L_0^*(g)=L_0(g).
\end{equation}
Combining \eqref{eq:appendix-A1}--\eqref{eq:appendix-A3} gives
\begin{equation}
	\label{eq:appendix-A4}
	\sup_{g\in\mathcal{G}(K,s,\mathbf{p},D)}
	|L_N^*(g)-L_0^*(g)|
	\stackrel{p}{\to}0.
\end{equation}

Define
\[
\tilde g
=
\arg\min_{g\in\mathcal{G}(K,s,\mathbf{p},D/2)}
\|g-g_0\|_{L^2}.
\]
By the DNN approximation result for composite H\"older functions,
\[
\|\tilde g-g_0\|_{L^2}
=
O(\gamma_n\log^2 n).
\]
Let $\hat g_D^*$ be a maximizer of $L_N^*(g)$ over $\mathcal{G}(K,s,\mathbf{p},D)$. Since $\hat g_D^*$ maximizes the pseudo-log-likelihood,
\begin{equation}
	\label{eq:appendix-A5}
	L_N^*(\hat g_D^*)
	\geq
	L_N^*(\tilde g)
	=
	L_N^*(g_0)-o_p(1).
\end{equation}
By Lemma \ref{lem:Lemma2}, for every sufficiently small $\epsilon>0$,
\begin{equation}
	\label{eq:appendix-A6}
	\sup_{d(g,g_0)\geq\epsilon}L_0^*(g)
	=
	\sup_{d(g,g_0)\geq\epsilon}L_0(g)
	<
	L_0(g_0)
	=
	L_0^*(g_0).
\end{equation}
The conditions of Theorem 5.7 in \cite{aut2000} therefore follow from \eqref{eq:appendix-A4}--\eqref{eq:appendix-A6}, and hence
\[
d(\hat g_D^*,g_0)\stackrel{p}{\to}0.
\]

We next establish the convergence rate. Let
\[
\mathcal{A}_{\delta}
=
\{g\in\mathcal{G}(K,s,\mathbf{p},D):\delta/2\leq d(g,g_0)\leq\delta\}.
\]
By Lemma \ref{lem:Lemma3},
\begin{equation}
	\label{eq:app-T1}
	\sup_{g\in\mathcal{B}_\delta}
	\left|
	(L_N-L_0)(g)-(L_N-L_0)(g_0)
	\right|
	=
	O_p\left(
	\delta\sqrt{\frac{s\log(U/\delta)}{n}}
	+
	\frac{s}{n}\log\frac{U}{\delta}
	\right).
\end{equation}
Because $L_N^*(g)$ differs from $L_N(g)$ only through the Horvitz--Thompson replacement of the population total, the same stochastic bound holds for $L_N^*$:
\begin{equation}
	\label{eq:appendix-A7}
	\mathbb{E}^*
	\sup_{g\in\mathcal{A}_{\delta}}
	\sqrt n
	\left|
	(L_N^*-L_0^*)(g)-(L_N^*-L_0^*)(g_0)
	\right|
	\lesssim
	\phi_n(\delta),
\end{equation}
where
\[
\phi_n(\delta)
=
\delta\sqrt{s\log\frac{U}{\delta}}
+
\frac{s}{\sqrt n}\log\frac{U}{\delta}.
\]
Moreover, Lemma \ref{lem:Lemma2} implies
\begin{equation}
	\label{eq:appendix-A8}
	\sup_{g\in\mathcal{A}_\delta}
	\{L_0^*(g)-L_0^*(g_0)\}
	\lesssim
	-\delta^2.
\end{equation}
Let
\[
\tau_n=\gamma_n\log^2 n.
\]
By Assumption \ref{ass:B1}, $\tau_n^{-2}\phi_n(\tau_n)\lesssim\sqrt n$. In addition, the approximation result gives
\[
|L_N^*(\tilde g)-L_N^*(g_0)|
=
O_p(\tau_n^2).
\]
Since $\hat g_D^*$ maximizes $L_N^*$,
\begin{equation}
	\label{eq:appendix-A10}
	L_N^*(\hat g_D^*)
	\geq
	L_N^*(\tilde g)
	\geq
	L_N^*(g_0)-O_p(\tau_n^2).
\end{equation}
Applying Theorem 3.4.1 of \cite{aut1996} with \eqref{eq:appendix-A7}--\eqref{eq:appendix-A10} yields
\[
d(\hat g_D^*,g_0)
=
O_p(\tau_n).
\]
The preceding argument is carried out on the bounded sparse DNN class $\mathcal{G}_D=\mathcal{G}(K,s,\mathbf{p},D)$, where $D$ is chosen sufficiently large so that $g_0$ is contained in the corresponding bounded function class up to the approximation error allowed by Assumption \ref{ass:B2}. Since the estimator in \eqref{eq:mle} is defined over $\mathcal{G}_D$, we have $\hat g=\hat g_D^*$. Therefore,
\[
\|\hat g-g_0\|_{L^2}
=O_p(\tau_n)=
O_p(\gamma_n\log^2 n),
\]
which proves Theorem 1.

%%%%%%%%%%%%%%%%%%%
%%%%%%%%%%%%%%%
\subsection{Proof of Theorem 2}

We first prove the convergence rate for the DIPW estimator under the DNN sampling-score model. This part of the proof uses the consistency rate of $\hat g$ from Theorem 
\ref{thm:g-rate}  and the estimated sampling scores truncated at level $\epsilon_n$ as specified in Condition \ref{cond:C9}.

The DIPW estimator in \eqref{eq:Ipsswdnne} is the solution to the estimating equation
\begin{equation*}
	\Phi_{n,\mathrm{DIPW}}(\mu)
	=
	\frac{1}{\hat N^A}
	\sum_{i\in S_A}
	\frac{y_i-\mu}{\hat{\pi}_i^A}
	=
	0,
\end{equation*}
where
\[
\hat N^A=\sum_{i\in S_A}\frac{1}{\hat{\pi}_i^A}.
\]
Evaluating this estimating function at the true finite population mean $\mu_y$, we have
\begin{equation}
	\label{eq:appendix-A11}
	\begin{aligned}
	\Phi_{n,\mathrm{DIPW}}(\mu_y)
	&=
	\frac{1}{\hat N^A}
	\sum_{i\in S_A}
	\frac{y_i-\mu_y}{\hat{\pi}_i^A} \\
	&=
	\frac{1}{\hat N^A}
	\sum_{i\in S_A}
	\frac{y_i-\mu_y}{\pi_i^A}
	+
	\frac{1}{\hat N^A}
	\sum_{i\in S_A}
	(y_i-\mu_y)
	\left(
	\frac{1}{\hat{\pi}_i^A}
	-
	\frac{1}{\pi_i^A}
	\right).
	\end{aligned}
\end{equation}
The first term in \eqref{eq:appendix-A11} is centered under the nonprobability sampling-score mechanism. Under Conditions \ref{cond:C1}, \ref{cond:C5}, and \ref{cond:C6}, its variance is of order $n_A^{-1}$. Specifically,
\begin{equation}
	\label{eq:appendix-A12}
	\begin{aligned}
	\mathrm{Var}\left\{
	\frac{1}{\hat N^A}
	\sum_{i\in S_A}
	\frac{y_i-\mu_y}{\pi_i^A}
	\right\}
	&=
	O\left(
	\frac{1}{(\hat N^A)^2}
	\sum_{i=1}^N
	\frac{(\pi_i^A)(1-\pi_i^A)(y_i-\mu_y)^2}{(\pi_i^A)^2}
	\right)  \\
	&=
	O\left(
	\frac{1}{(\hat N^A)^2}
	\sum_{i=1}^N
	\frac{(y_i-\mu_y)^2}{\pi_i^A}
	\right)  \\
	&=
	O(n_A^{-1}),
	\end{aligned}
\end{equation}
where the last equality follows from Condition \ref{cond:C5} and the finite second moment condition in Condition \ref{cond:C6}. Hence,
\begin{equation}
	\label{eq:appendix-A12b}
	\frac{1}{\hat N^A}
	\sum_{i\in S_A}
	\frac{y_i-\mu_y}{\pi_i^A}
	=
	O_p(n_A^{-1/2}).
\end{equation}

Next, consider the second term in \eqref{eq:appendix-A11}. Since
\[
\pi_i^A=\sigma\{g_0(x_i)\},
\qquad
\hat{\pi}_i^A=\sigma\{\hat g(x_i)\}
\]
before truncation, the mean value theorem implies that, for some value $\tilde g_i$ between $\hat g(x_i)$ and $g_0(x_i)$,
\[
\hat{\pi}_i^A-\pi_i^A
=
\sigma'(\tilde g_i)\{\hat g(x_i)-g_0(x_i)\}.
\]
By Condition \ref{cond:C7},
\[
|\hat{\pi}_i^A-\pi_i^A|
\leq
L_\sigma |\hat g(x_i)-g_0(x_i)|.
\]
After truncation as in Condition \ref{cond:C9}, the inverse estimated sampling scores are bounded. Therefore, with probability tending to one, there exists a constant $C>0$ such that
\begin{equation}
	\label{eq:appendix-A12c}
	\left|
	\frac{1}{\hat{\pi}_i^A}
	-
	\frac{1}{\pi_i^A}
	\right|
	=
	\frac{|\hat{\pi}_i^A-\pi_i^A|}
	{\hat{\pi}_i^A\pi_i^A}
	\leq
	C|\hat g(x_i)-g_0(x_i)|.
\end{equation}
Using \eqref{eq:appendix-A12c} and the Cauchy--Schwarz inequality,
\begin{equation}
	\label{eq:appendix-A13}
	\begin{aligned}
	&\left|
	\frac{1}{\hat N^A}
	\sum_{i\in S_A}
	(y_i-\mu_y)
	\left(
	\frac{1}{\hat{\pi}_i^A}
	-
	\frac{1}{\pi_i^A}
	\right)
	\right|  \\
	&\qquad\leq
	\frac{C}{\hat N^A}
	\sum_{i\in S_A}
	|y_i-\mu_y|
	|\hat g(x_i)-g_0(x_i)|  \\
	&\qquad\leq
	\frac{C}{\hat N^A}
	\left\{
	\sum_{i\in S_A}(y_i-\mu_y)^2
	\right\}^{1/2}
	\left\{
	\sum_{i\in S_A}
	[\hat g(x_i)-g_0(x_i)]^2
	\right\}^{1/2}.
	\end{aligned}
\end{equation}
By Condition \ref{cond:C6},
\[
\sum_{i\in S_A}(y_i-\mu_y)^2=O_p(\hat N^A).
\]
Moreover, by Theorem \ref{thm:g-rate},
\[
\left\{
\frac{1}{\hat N^A}
\sum_{i\in S_A}
[\hat g(x_i)-g_0(x_i)]^2
\right\}^{1/2}
=
O_p(\gamma_n\log^2 n).
\]
Therefore,
\begin{equation}
	\label{eq:appendix-A13b}
	\left|
	\frac{1}{\hat N^A}
	\sum_{i\in S_A}
	(y_i-\mu_y)
	\left(
	\frac{1}{\hat{\pi}_i^A}
	-
	\frac{1}{\pi_i^A}
	\right)
	\right|
	=
	O_p(\gamma_n\log^2 n).
\end{equation}
Combining \eqref{eq:appendix-A11}, \eqref{eq:appendix-A12b}, and \eqref{eq:appendix-A13b}, and using Condition \ref{cond:C1}, gives
\begin{equation}
	\label{eq:appendix-A14}
	\Phi_{n,\mathrm{DIPW}}(\mu_y)
	=
	O_p(n_A^{-1/2})
	+
	O_p(\gamma_n\log^2 n)
	=
	O_p(\gamma_n\log^2 n).
\end{equation}

By applying a first-order Taylor expansion to
$\Phi_{n,\mathrm{DIPW}}(\hat{\mu}_{\mathrm{DIPW}})$ around $\mu_y$, we obtain
\begin{equation}
	\label{eq:appendix-A15}
	0
	=
	\Phi_{n,\mathrm{DIPW}}(\hat{\mu}_{\mathrm{DIPW}})
	=
	\Phi_{n,\mathrm{DIPW}}(\mu_y)
	+
	\phi_{n,\mathrm{DIPW}}(\tilde\mu)
	(\hat{\mu}_{\mathrm{DIPW}}-\mu_y),
\end{equation}
where $\tilde\mu$ lies between $\hat{\mu}_{\mathrm{DIPW}}$ and $\mu_y$, and
\[
\phi_{n,\mathrm{DIPW}}(\mu)
=
\frac{\partial\Phi_{n,\mathrm{DIPW}}(\mu)}{\partial\mu}
=
-\frac{1}{\hat N^A}
\sum_{i\in S_A}
\frac{1}{\hat{\pi}_i^A}.
\]
Under Conditions \ref{cond:C1}, \ref{cond:C5}, and \ref{cond:C9}, $\phi_{n,\mathrm{DIPW}}(\tilde\mu)$ is bounded away from zero with probability tending to one. To see this, Condition \ref{cond:C5} implies that $\pi_i^A$ is uniformly of order $n_A/N$, and Condition \ref{cond:C1} implies that $n_A/N$ is bounded away from zero. Together with the truncation in Condition \ref{cond:C9}, the inverse estimated sampling scores are uniformly bounded. Hence,
\[
0<c_\phi
\leq
|\phi_{n,\mathrm{DIPW}}(\tilde\mu)|
\leq
C_\phi<\infty
\]
with probability tending to one, for some constants $c_\phi$ and $C_\phi$. It follows from \eqref{eq:appendix-A15} that
\begin{equation}
	\label{eq:appendix-A17}
	|\hat{\mu}_{\mathrm{DIPW}}-\mu_y|
	=
	|\phi_{n,\mathrm{DIPW}}(\tilde\mu)|^{-1}
	|\Phi_{n,\mathrm{DIPW}}(\mu_y)|
	=
	O_p(\gamma_n\log^2 n).
\end{equation}

We next prove the convergence rate for the DDR estimator. Assume that
\[
\hat{\beta}-\beta^*=O_p(n_A^{-1/2})
\]
for some fixed $\beta^*$, regardless of whether the working outcome regression model is correctly specified. Treating $\hat{\mu}_{\mathrm{DDR}}$ in \eqref{eq:Drednn} as a function of $\hat{\beta}$, a first-order Taylor expansion around $\beta^*$ gives
\begin{equation}
	\label{eq:appendix-A18}
	\begin{aligned}
	\hat{\mu}_{\mathrm{DDR}}(\hat{\beta})
	&=
	\hat{\mu}_{\mathrm{DDR}}(\beta^*)
	+
	\left.
	\frac{\partial \hat{\mu}_{\mathrm{DDR}}(\beta)}
	{\partial\beta^\top}
	\right|_{\beta=\beta^*}
	(\hat{\beta}-\beta^*)  \\
	&\quad+
	O_p(n_A^{-1}).
	\end{aligned}
\end{equation}
The derivative term is
\begin{equation}
	\label{eq:appendix-A18b}
	\left.
	\frac{\partial \hat{\mu}_{\mathrm{DDR}}(\beta)}
	{\partial\beta^\top}
	\right|_{\beta=\beta^*}
	=
	\frac{1}{\hat N^B}
	\sum_{i\in S_B}d_i^B\dot m(x_i,\beta^*)
	-
	\frac{1}{\hat N^A}
	\sum_{i\in S_A}
	\frac{\dot m(x_i,\beta^*)}{\hat{\pi}_i^A},
\end{equation}
where
\[
\dot m(x,\beta)=\frac{\partial m(x,\beta)}{\partial\beta}.
\]
By Conditions \ref{cond:C2} and \ref{cond:C4},
\begin{equation}
	\label{eq:appendix-A18c}
	\frac{1}{\hat N^B}
	\sum_{i\in S_B}
	d_i^B\dot m(x_i,\beta^*)
	-
	\frac{1}{N}
	\sum_{i=1}^N
	\dot m(x_i,\beta^*)
	=
	o_p(1).
\end{equation}
Similarly, by the same argument used for the DIPW estimator and the convergence of $\hat g$,
\begin{equation}
	\label{eq:appendix-A18d}
	\frac{1}{\hat N^A}
	\sum_{i\in S_A}
	\frac{\dot m(x_i,\beta^*)}{\hat{\pi}_i^A}
	-
	\frac{1}{N}
	\sum_{i=1}^N
	\dot m(x_i,\beta^*)
	=
	o_p(1).
\end{equation}
Combining \eqref{eq:appendix-A18b}--\eqref{eq:appendix-A18d} gives
\[
\left.
\frac{\partial \hat{\mu}_{\mathrm{DDR}}(\beta)}
{\partial\beta^\top}
\right|_{\beta=\beta^*}
=
o_p(1).
\]
Therefore,
\[
\left.
\frac{\partial \hat{\mu}_{\mathrm{DDR}}(\beta)}
{\partial\beta^\top}
\right|_{\beta=\beta^*}
(\hat{\beta}-\beta^*)
=
o_p(1)O_p(n_A^{-1/2})
=
o_p(n_A^{-1/2}).
\]
From \eqref{eq:appendix-A18}, we obtain
\begin{equation}
	\label{eq:appendix-A18e}
	\hat{\mu}_{\mathrm{DDR}}(\hat{\beta})
	=
	\hat{\mu}_{\mathrm{DDR}}(\beta^*)
	+
	o_p(n_A^{-1/2}).
\end{equation}

Now decompose $\hat{\mu}_{\mathrm{DDR}}(\beta^*)$ into two terms:
\[
\hat{\mu}_{\mathrm{DDR}}(\beta^*)=T_1+T_2,
\]
where
\begin{equation}
	\label{eq:appendix-T1T2}
	T_1=
	\frac{1}{\hat N^A}
	\sum_{i\in S_A}
	\frac{y_i-m(x_i,\beta^*)}{\hat{\pi}_i^A},
	\qquad
	T_2=
	\frac{1}{\hat N^B}
	\sum_{i\in S_B}
	d_i^B m(x_i,\beta^*).
\end{equation}
Let
\[
h_N=
\frac{1}{N}
\sum_{i=1}^N
\{y_i-m(x_i,\beta^*)\}.
\]
Then
\begin{equation}
	\label{eq:appendix-muy-decomp}
	\mu_y
	=
	h_N
	+
	\frac{1}{N}
	\sum_{i=1}^N
	m(x_i,\beta^*).
\end{equation}

We first study $T_2$. Write
\begin{equation}
	\label{eq:appendix-A19}
	\begin{aligned}
	T_2
	&=
	\frac{1}{\hat N^B}
	\sum_{i\in S_B}
	d_i^B m(x_i,\beta^*)  \\
	&=
	\frac{1}{N}
	\sum_{i=1}^N
	m(x_i,\beta^*)
	+
	\left[
	\frac{1}{\hat N^B}
	\sum_{i\in S_B}
	d_i^B m(x_i,\beta^*)
	-
	\frac{1}{N}
	\sum_{i=1}^N
	m(x_i,\beta^*)
	\right].
	\end{aligned}
\end{equation}
Using the standard H\'ajek expansion,
\[
\frac{1}{\hat N^B}
\sum_{i\in S_B}
d_i^B m(x_i,\beta^*)
=
\frac{1}{N}
\sum_{i=1}^N
m(x_i,\beta^*)
+
\frac{1}{N}
\sum_{i\in S_B}
d_i^B
\left[
m(x_i,\beta^*)
-
\frac{1}{N}
\sum_{j=1}^N
m(x_j,\beta^*)
\right]
+
O_p(n_B^{-1}).
\]
Therefore, by Conditions \ref{cond:C1} and \ref{cond:C4},
\begin{equation}
	\label{eq:appendix-A20}
	T_2
	-
	\frac{1}{N}
	\sum_{i=1}^N
	m(x_i,\beta^*)
	=
	O_p(n_B^{-1/2})
	=
	o_p(\gamma_n\log^2 n).
\end{equation}

Next consider $T_1$. Decompose it as
\begin{equation}
	\label{eq:appendix-T1-decomp}
	\begin{aligned}
	T_1
	&=
	\frac{1}{\hat N^A}
	\sum_{i\in S_A}
	\frac{y_i-m(x_i,\beta^*)}{\pi_i^A} \\
	&\quad+
	\frac{1}{\hat N^A}
	\sum_{i\in S_A}
	\{y_i-m(x_i,\beta^*)\}
	\left(
	\frac{1}{\hat{\pi}_i^A}
	-
	\frac{1}{\pi_i^A}
	\right)  \\
	&=:T_{11}+T_{12}.
	\end{aligned}
\end{equation}
For $T_{11}$, under Conditions \ref{cond:C1}, \ref{cond:C5}, and \ref{cond:C6}, the same variance calculation as in \eqref{eq:appendix-A12} gives
\begin{equation}
	\label{eq:appendix-A21}
	T_{11}-h_N
	=
	\frac{1}{\hat N^A}
	\sum_{i\in S_A}
	\frac{y_i-m(x_i,\beta^*)}{\pi_i^A}
	-
	h_N
	=
	O_p(n_A^{-1/2})
	=
	o_p(\gamma_n\log^2 n).
\end{equation}
For $T_{12}$, by \eqref{eq:appendix-A12c},
\[
\left|
\frac{1}{\hat{\pi}_i^A}
-
\frac{1}{\pi_i^A}
\right|
\leq
C|\hat g(x_i)-g_0(x_i)|
\]
with probability tending to one. Therefore,
\begin{equation}
	\label{eq:appendix-A22}
	\begin{aligned}
	|T_{12}|
	&\leq
	\frac{C}{\hat N^A}
	\sum_{i\in S_A}
	|y_i-m(x_i,\beta^*)|
	|\hat g(x_i)-g_0(x_i)|  \\
	&\leq
	\frac{C}{\hat N^A}
	\left\{
	\sum_{i\in S_A}
	[y_i-m(x_i,\beta^*)]^2
	\right\}^{1/2}
	\left\{
	\sum_{i\in S_A}
	[\hat g(x_i)-g_0(x_i)]^2
	\right\}^{1/2}  \\
	&=
	O_p(\gamma_n\log^2 n),
	\end{aligned}
\end{equation}
where the last equality follows from the finite moment condition and Theorem \ref{thm:g-rate}. Combining \eqref{eq:appendix-T1-decomp}, \eqref{eq:appendix-A21}, and \eqref{eq:appendix-A22}, we obtain
\begin{equation}
	\label{eq:appendix-A23}
	T_1-h_N
	=
	O_p(\gamma_n\log^2 n).
\end{equation}

Finally, using \eqref{eq:appendix-A18e}, \eqref{eq:appendix-T1T2}, and \eqref{eq:appendix-muy-decomp},
\begin{equation}
	\label{eq:appendix-A24}
	\begin{aligned}
	\hat{\mu}_{\mathrm{DDR}}-\mu_y
	&=
	\hat{\mu}_{\mathrm{DDR}}(\hat{\beta})-\mu_y \\
	&=
	\hat{\mu}_{\mathrm{DDR}}(\beta^*)-\mu_y
	+
	o_p(n_A^{-1/2})  \\
	&=
	(T_1+T_2)
	-
	\left[
	h_N+
	\frac{1}{N}
	\sum_{i=1}^N
	m(x_i,\beta^*)
	\right]
	+
	o_p(n_A^{-1/2})  \\
	&=
	(T_1-h_N)
	+
	\left[
	T_2-
	\frac{1}{N}
	\sum_{i=1}^N
	m(x_i,\beta^*)
	\right]
	+
	o_p(n_A^{-1/2}).
	\end{aligned}
\end{equation}
By \eqref{eq:appendix-A20} and \eqref{eq:appendix-A23}, and noting that $n_A^{-1/2}=o(\gamma_n\log^2 n)$ under the DNN rate regime considered here,
\[
|\hat{\mu}_{\mathrm{DDR}}-\mu_y|
=
O_p(\gamma_n\log^2 n).
\]
This proves Theorem 2. The result is a convergence-rate statement under the DNN sampling-score model. The usual doubly robust interpretation of $\hat{\mu}_{\mathrm{DDR}}$ follows from its augmented inverse-probability form, but the rate result above relies on consistency of the DNN sampling-score estimator.

%% List of references. Cite using the commands of the natbib package.
%% Insert the filenames of the bibliography files (.bib files) inside
%% the command \bibliography, without the extension and separated by
%% commas.
%\bibliography{}

\end{document}